\documentclass[10pt]{amsart}

\usepackage{amsmath,amscd}
\usepackage{amssymb}
\usepackage{mathtools}
\usepackage{url}

\usepackage[all]{xy}

\usepackage{times}

\usepackage[T1]{fontenc}
\usepackage{graphicx}
\usepackage[svgnames]{xcolor}
\usepackage{framed}

\usepackage{tikz}

\usepackage{enumitem,kantlipsum}

\author{Liran Shaul}
\address{Department of Algebra, Faculty of Mathematics and Physics, Charles University in Prague, Sokolovsk\'a 83, 186 75 Praha, Czech Republic}

\email{shaul@karlin.mff.cuni.cz}
\newtheorem{thm}[equation]{Theorem}
\newtheorem{cthm}{Theorem}
\newtheorem{ccor}[cthm]{Corollary}

\newtheorem{cor}[equation]{Corollary}
\newtheorem{prop}[equation]{Proposition}
\newtheorem{lem}[equation]{Lemma}

\theoremstyle{definition}
\newtheorem{dfn}[equation]{Definition}
\newtheorem{rem}[equation]{Remark}

\newtheorem{exa}[equation]{Example}

\newcommand{\surj}{\twoheadrightarrow}

\newcommand{\opn}{\operatorname}
\newcommand{\cat}[1]{\operatorname{\mathsf{#1}}}

\newcommand{\mfrak}[1]{\mathfrak{#1}}

\newcommand{\mrm}[1]{\mathrm{#1}}
\newcommand{\mbb}[1]{\mathbb{#1}}

\renewcommand{\k}{\Bbbk}

\newcommand{\K}{\mbb{K} \hspace{0.05em}}

\renewcommand{\a}{\mfrak{a}}

\newcommand{\m}{\mfrak{m}}
\newcommand{\n}{\mfrak{n}}
\newcommand{\p}{\mfrak{p}}
\newcommand{\q}{\mfrak{q}}

\newcommand{\flatdim}{\operatorname{flat\,dim}}
\newcommand{\depth}{\operatorname{depth}}
\newcommand{\amp}{\operatorname{amp}}
\def\skewtimes{\ltimes\!}
\newcommand{\lcdim}{\operatorname{lc.\dim}}

\newcommand{\cdg}{\cat{DGR}}
\newcommand{\ho}{\mrm{Ho}}

\begin{document}

\title{Koszul complexes over Cohen-Macaulay rings}

\begin{abstract}
We prove a Cohen-Macaulay version of a result by Avramov-Golod and Frankild-J{\o}rgensen about Gorenstein rings,
showing that if a noetherian ring $A$ is Cohen-Macaulay, and $a_1,\dots,a_n$ is any sequence of elements in $A$,
then the Koszul complex $K(A;a_1,\dots,a_n)$ is a Cohen-Macaulay DG-ring.
We further generalize this result,
showing that it also holds for commutative DG-rings.
In the process of proving this, 
we develop a new technique to study the dimension theory of a noetherian ring $A$,
by finding a Cohen-Macaulay DG-ring $B$ such that $\mrm{H}^0(B) = A$,
and using the Cohen-Macaulay structure of $B$ to deduce results about $A$.
As application, we prove that if $f:X \to Y$ is a morphism of schemes,
where $X$ is Cohen-Macaulay and $Y$ is nonsingular, 
then the homotopy fiber of $f$ at every point is Cohen-Macaulay.
As another application, we generalize the miracle flatness theorem.
Generalizations of these applications to derived algebraic geometry are also given.
\end{abstract}


\setcounter{tocdepth}{1}
\setcounter{section}{-1}

\maketitle

\section{Introduction}
\numberwithin{equation}{section}

Given a commutative noetherian ring $A$, and a finite sequence of elements $a_1,\dots,a_n$ in $A$,
a basic construction in commutative algebra is the quotient ring $A/(a_1,\dots,a_n)$.
In homological contexts, this operation is particularly well behaved when the sequence $a_1,\dots,a_n$ is an $A$-regular sequence.
For instance:\\
(*) If $A$ is a Gorenstein or a Cohen-Macaulay local ring and this sequence is regular,
then the quotient ring $A/(a_1,\dots,a_n)$ is also Gorenstein or Cohen-Macaulay.

A well known principle in homological algebra is that often results about functors that hold only under nice homological conditions,
hold unconditionally when one works with derived functors.

The aim of this paper is to show how this principle holds regarding (*).
We will show that the derived quotient of a Gorenstein or a Cohen-Macaulay ring with respect to any finite sequence of elements is again Gorenstein or Cohen-Macaulay.
To do this, let us first explain how to derive the operation $A \mapsto A/(a_1,\dots,a_n)$.

Considering $A$ as a $\mathbb{Z}[x_1,\dots,x_n]$-algebra
by letting $x_i \mapsto a_i$ for any $1 \le i \le n$,
we have a ring isomorphism
\[
A \otimes_{\mathbb{Z}[x_1,\dots,x_n]} \mathbb{Z} \cong A/(a_1,\dots,a_n).
\]
With this realization of $A/(a_1,\dots,a_n)$,
it is clear how to derive this operation,
by taking the derived functor of the tensor product.
It follows that we may consider
\begin{equation}\label{eqn:basicKoszul}
A \otimes^{\mrm{L}}_{\mathbb{Z}[x_1,\dots,x_n]} \mathbb{Z}
\end{equation}
as a commutative non-positive DG-ring which represents this derived quotient.
To compute (\ref{eqn:basicKoszul}) observe that the $\mathbb{Z}[x_1,\dots,x_n]$-algebra $\mathbb{Z}$
has a concrete flat DG-algebra resolution using the Koszul complex $K(\mathbb{Z}[x_1,\dots,x_n];x_1,\dots,x_n)$.
It follows from the base change property of the Koszul complex that
\[
A \otimes^{\mrm{L}}_{\mathbb{Z}[x_1,\dots,x_n]} \mathbb{Z} \cong
A \otimes_{\mathbb{Z}[x_1,\dots,x_n]} K(\mathbb{Z}[x_1,\dots,x_n];x_1,\dots,x_n) \cong
K(A;a_1,\dots,a_n).
\]

We see that one may consider the Koszul complex $K(A;a_1,\dots,a_n)$ as the derived functor of $A \mapsto A/(a_1,\dots,a_n)$.
Observe further that $\mrm{H}^0(K(A;a_1,\dots,a_n)) = A/(a_1,\dots,a_n)$,
as one would expect from a derived functor,
and that $K(A;a_1,\dots,a_n)$ is isomorphic to $A/(a_1,\dots,a_n)$ if $a_1,\dots,a_n$ is an $A$-regular sequence.	

The Koszul complex over a commutative ring has the structure of a commutative non-positive DG-ring.
It turns out that one can generalize the Koszul complex construction over a commutative DG-ring $A$,
but instead of depending on elements of $A$,
it is associated to a finite sequence of elements in $\mrm{H}^0(A)$ (see Definition \ref{def:KosDG} below).
The main result of this paper states:

\begin{cthm}\label{cthm:main}
\leavevmode
\begin{enumerate}[wide, labelindent=0pt]
\item Let $A$ be a Cohen-Macaulay ring, and let $a_1,\dots,a_n$ be a finite sequence of elements in $A$.
Then the Koszul complex $K(A;a_1,\dots,a_n)$ is a Cohen-Macaulay DG-ring.
\item More generally, let $A$ be a Cohen-Macaulay DG-ring with constant amplitude,
and let $\bar{a}_1,\dots,\bar{a}_n \in \mrm{H}^0(A)$.
Then the Koszul complex $K(A;\bar{a}_1,\dots,\bar{a}_n)$ is a Cohen-Macaulay DG-ring.
\item Let $A$ be commutative noetherian DG-ring with bounded cohomology.
Given any sequence $\bar{a}_1,\dots,\bar{a}_n \in \mrm{H}^0(A)$,
if $A$ is Gorenstein then $K(A;\bar{a}_1,\dots,\bar{a}_n)$ is a Gorenstein DG-ring.
The converse holds if $(\bar{a}_1,\dots,\bar{a}_n) \subseteq \opn{rad}(\mrm{H}^0(A))$.
\end{enumerate}
\end{cthm}

Theorem \ref{cthm:main}(1) and Theorem \ref{cthm:main}(2) are completely new, 
and no other result of this kind appeared before in the literature.
The vast majority of this paper develops the tools needed to prove this.
The \textbf{constant amplitude} assumption in Theorem \ref{cthm:main}(2) means that
\[
\opn{Supp}(\mrm{H}^{\inf(A)}(A)) = \opn{Spec}(\mrm{H}^0(A)).
\]
This assumption is necessary.
In Example \ref{exa:counter} we show that the result is false without this assumption.
This condition is automatically satisfied if $\opn{Spec}(\mrm{H}^0(A))$ is irreducible.
See Remark \ref{rem:condition} for a discussion about this assumption.

A particular case of Theorem \ref{cthm:main}(3) was first proved by Avramov and Golod in \cite{AG},
this was later generalized by Frankild and J{\o}rgensen in \cite[Theorem 4.9]{FJ}.
There, it was proved that if $(A,\m)$ is a noetherian local ring,
and if $a_1,\dots,a_n \in \m$,
then $A$ is Gorenstein if and only if $K(A;a_1,\dots,a_n)$ is Gorenstein.
Our result generalizes the theorem of Frankild and J{\o}rgensen in two ways,
allowing $A$ to be non-local, and moreover, letting $A$ be a DG-ring instead of a ring.

The paper is organized as follows. 
Section \ref{sec:pre} consists of various preliminaries, 
where we recall basics about commutative DG-rings, Gorenstein DG-rings and Cohen-Macaulay DG-rings.

In Section \ref{sec:Koszul} we introduce and study, following previous work of Minamoto, the Koszul DG-module over a commutative DG-ring.
This section extends previously known results about Koszul complexes to the DG-setting.
It contains one relatively difficult result, Theorem \ref{thm:koszulCompletion},
which shows that the Koszul complex of a noetherian local DG-ring commutes with derived adic completion.
The proof is rather difficult because our noetherian assumption is only on the level of cohomology,
so one cannot realize adic completion using a tensor product. 
This implies that the usual proof that the Koszul complex commutes with adic completion because of base change reasons does not work.
Instead we take a different strategy, realizing the derived adic completion using a DG version of the Matlis theory of injective modules.

Section \ref{sec:dim} is devoted to the study of dimension theory over Cohen-Macaulay DG-rings.
In Theorem \ref{thm:depth-bound} we generalize a classical result from commutative algebra to the DG setting,
showing that over a noetherian local DG-ring,
the sequential depth of a DG-module $M$ is bounded by the coheight of any associated prime ideal of $M$.
Using this, in Theorem \ref{thm:seq-depth}, the main result of that section,
we obtain an explicit formula for the sequential depth of any ideal in a Cohen-Macaulay local DG-ring with constant amplitude.

The results of Section \ref{sec:dim} demonstrate a new technique to study the dimension theory of a noetherian ring $A$:
find a Cohen-Macaulay DG-ring $B$ such that $\mrm{H}^0(B) = A$,
and then deduce results from the Cohen-Macaulay structure of $B$ about the dimension theory of $A$.
A corollary of Theorem \ref{cthm:main}(1) is that one can always find such a $B$, 
provided that $A$ is a quotient of some Cohen-Macaulay ring.

In Section \ref{sec:KosCMDG}, using all these tools, 
we prove Theorem \ref{cthm:main}.
The strategy to prove Theorem \ref{cthm:main}(2) is to, first,
using the results of Section \ref{sec:Koszul},
reduce to the case where $(A,\bar{\m})$ is local and derived $\bar{\m}$-adically complete.
This implies that $A$, and hence also its Koszul complex, have dualizing DG-modules.
We then use the explicit formulas obtained in Section \ref{sec:dim} to compute the amplitude of the Koszul complex over $A$,
and of its dualizing DG-module, showing they are equal.

In the final Section \ref{sec:APP} we discuss some applications of Theorem \ref{cthm:main}.
We recall the notion of a homotopy fiber of a local homomorphism,
and show:

\begin{ccor}
Let $f:X \to Y$ be a morphism of schemes,
and assume that $X$ is Cohen-Macaulay and that $Y$ is nonsingular.
Then the homotopy fiber of $f$ at every point is a Cohen-Macaulay DG-ring.
\end{ccor}

This is a particular case of a more general result we prove,
that is also valid for $X$ being a derived scheme instead of a scheme.
Our next application is a generalization of the miracle flatness theorem.
The classical miracle flatness theorem states that a local homomorphism $\varphi$ from a regular local ring $A$ to a Cohen-Macaulay local ring $B$ is flat if and only if the dimension of the fiber of $\varphi$ is equal to $\dim(B) - \dim(A)$.
In Corollary \ref{cor:miracle},
we generalize this in two ways, 
obtaining an explicit formula for the flat dimension of $B$ over $A$ in terms of the dimension of the fiber,
and furthermore, allowing $B$ to be a Cohen-Macaulay DG-ring instead of a Cohen-Macaulay ring.
Our final application is a generalization of the following characterization of Cohen-Macaulay rings.
Recall that if a noetherian local ring $B$ contains a regular local ring $A$,
such that the extension $A\subseteq B$ is finite,
then $B$ is Cohen-Macaulay if and only if $B$ is a free $A$-module.
We show that the same holds in the DG-setting,
showing in Corollary \ref{cor:DGreg},
that in such a finite situation,
the Cohen-Macaulay DG-rings are exactly the DG-rings which have a free resolution over $A$ whose length is equal to its cohomological amplitude.

\section{Preliminaries}\label{sec:pre}

All rings in this paper are commutative and unital,
and in most cases they will also be noetherian.
For a noetherian ring $A$,
we denote by $\dim(A)$ its Krull dimension.
If $M$ is an $A$-module, 
we let $\dim(M)$ denote the Krull dimension of $M$,
and the support of $M$ is the set $\opn{Supp}(M) = \{\p \in \opn{Spec}(A) \mid M_{\p} \ne 0\}$.

\subsection{Commutative DG-rings}
The main objects of study in this paper are commutative non-positive DG-rings.
By definition, these are graded rings $A = \bigoplus_{n=-\infty}^0 A^n$,
equipped with a $\mathbb{Z}$-linear differential $d:A \to A$.
We will be using cohomological gradings, 
so that $d$ is of degree $+1$.
The fact that $A$ is commutative means that for all $a,b 
\in A$ we have that $b\cdot a = (-1)^{\deg(a)\cdot \deg(b)}\cdot a \cdot b$,
and moreover, if $\deg(a)$ is odd, then $a^2 = 0$.
The multiplication of $A$ and the differential $d$ satisfy a Leibnitz rule:
\[
d(a\cdot b) = d(a)\cdot b + (-1)^{\deg(a)}\cdot a \cdot d(b).
\]
All DG-rings in this paper are assumed to be non-positive and commutative.
A reference for commutative DG-rings and their derived categories is the recent book \cite{YeBook}.
Other helpful introductions to the theory of DG-rings include \cite[Section 1]{YeSq} and \cite{BS}.

The category of all (not necessarily commutative) non-positive DG-rings will be denoted by $\cdg$.
This has a natural Quillen model structure. Inverting quasi-isomorphisms in $\cdg$,
we obtain its homotopy category which we will denote by $\ho(\cdg)$.

Given a commutative DG-ring $\K$,
and two commutative DG-algebras $A,B$ over $\K$,
the derived tensor product $A\otimes^{\mrm{L}}_{\K} B$ is a well defined functor in $\ho(\cdg)$.
We will always represent it using a commutative DG-ring, 
as it is always possible to do that.

Given a commutative DG-ring $A$, 
DG-modules over $A$ are by definition $\mathbb{Z}$-graded $A$-modules $M$ with a differential,
which satisfy a Leibnitz rule. 
The derived category of all DG-modules over $A$ will be denoted by $\cat{D}(A)$.
This is a triangulated category.

For a commutative non-positive DG-ring $A$, its bottom cohomology $\mrm{H}^0(A)$ is a commutative ring.
There is a natural map of DG-rings $\pi_A:A \to \mrm{H}^0(A)$.
Moreover, its degree zero part $\pi_A^0:A^0 \to \mrm{H}^0(A)$ is a surjection of commutative rings.

\subsection{Finiteness conditions}

We say that a commutative DG-ring $A$ has bounded cohomology if $H^n(A) = 0$ for all $n \ll 0$.
If $M$ is any DG-module over $A$, and $n \in \mathbb{Z}$,
then $\mrm{H}^n(M)$ has the structure of an $\mrm{H}^0(A)$-module.
In particular $\mrm{H}^n(A)$ is also an $\mrm{H}^0(A)$-module for all $n < 0$.
We say that $A$ is noetherian (called cohomologically pseudo-noetherian in \cite{YeBook}) if the ring $\mrm{H}^0(A)$ is noetherian,
and for all $n<0$, the $\mrm{H}^0(A)$-module $\mrm{H}^n(A)$ is finitely generated.
The main focus in this paper will be about commutative noetherian DG-rings.

We say that a DG-module $M$ over a noetherian DG-ring $A$ has finitely generated cohomology if $n \in \mathbb{Z}$,
the $\mrm{H}^0(A)$-module $\mrm{H}^n(M)$ is finitely generated.
The full triangulated subcategory of $\cat{D}(A)$ consisting of DG-modules with finitely generated cohomology is denoted by $\cat{D}_{\mrm{f}}(A)$. 
A DG-module $M$ is called bounded below (respectively above) if $H^n(M) = 0$ for all $n \ll 0$ (resp. $n \gg 0$),
and is called bounded if it both bounded below and bounded above.
The full triangulated subcategories of $\cat{D}(A)$ consisting of DG-modules which are bounded below, bounded above and bounded are denoted by
$\cat{D}^{-}(A)$, $\cat{D}^{+}(A)$ and $\cat{D}^{\mrm{b}}(A)$.
If $A$ is noetherian, 
we set $\cat{D}^{-}_{\mrm{f}}(A) = \cat{D}^{-}(A) \cap \cat{D}_{\mrm{f}}(A)$, 
and similarly for the other boundedness conditions.

For a DG-module $M$, we associate the following numbers, 
called the infimum, the supremum and the amplitude of $M$:
\[
\inf(M) := \inf\{n \in \mathbb{Z} \mid \mrm{H}^n(M) \ne 0\},
\quad
\sup(M) := \sup\{n \in \mathbb{Z} \mid \mrm{H}^n(M) \ne 0\},
\]
and $\amp(M) := \sup(M) - \inf(M)$.
These are sometimes called the cohomological infimum, cohomological supremum and cohomological amplitude of $M$,
but we omit this adjective, as the entire paper is of cohomological nature.

\subsection{Localization}
Given a prime ideal $\bar{\p} \in \opn{Spec}(\mrm{H}^0(A))$,
we define the localization of $A$ with respect to $\bar{\p}$ as follows:
let $\p := (\pi^0_A)^{-1}(\bar{\p}) \in \opn{Spec}(A^0)$,
and set $A_{\bar{\p}}:= A\otimes_{A^0} A^0_{\p}$.
For $M \in \cat{D}(A)$, we similarly set
$M_{\bar{\p}}:= M\otimes^{\mrm{L}}_A A_{\bar{\p}}$.
It then holds that $\mrm{H}^n(M_{\bar{\p}}) \cong (\mrm{H}^n(M))_{\bar{\p}}$,
where the right hand side is the usual localization of the $\mrm{H}^0(A)$-module $\mrm{H}^n(M)$.
A noetherian DG-ring $A$ is called local if the noetherian ring $\mrm{H}^0(A)$ is a local ring.
In that case, if $\bar{\m}$ is the maximal ideal of $\mrm{H}^0(A)$,
we will say that $(A,\bar{\m})$ is a noetherian local DG-ring.
If $A$ is a noetherian DG-ring,
and if $\bar{\p} \in \opn{Spec}(\mrm{H}^0(A))$,
it follows that $(A_{\bar{\p}},\bar{\p}\cdot \mrm{H}^0(A_{\bar{\p}}))$ is a noetherian local DG-ring.

\subsection{Dualizing DG-modules}

Given a commutative noetherian DG-ring $A$,
following \cite{FIJ,Ye1},
we say that a DG-module $R \in \cat{D}_{\mrm{f}}(A)$ is a dualizing DG-module if $R$ has finite injective dimension over $A$,
in the sense of \cite[Definition 12.4.8(2)]{YeBook},
and moreover, the natural map $A \to \mrm{R}\opn{Hom}_A(R,R)$ is an isomorphism.
In case $A$ is a ring, this is called a dualizing complex over $A$,
a notion introduced by Grothendieck.
If $(A,\bar{\m})$ is a noetherian local DG-ring,
and if $R,S$ are dualizing DG-modules over it,
then by \cite[Corollary 7.16]{Ye1},
there is some $n \in \mathbb{Z}$ such that $R \cong S[n]$.
If $A$ is a noetherian DG-ring with bounded cohomology,
and $R$ is a dualizing DG-module over $A$,
by \cite[Theorem 4.1(2)]{ShCM}, 
there is an inequality $\amp(R) \ge \amp(A)$.

\subsection{Gorenstein DG-rings}

Following \cite{AF,FI,FIJ},
we say that a noetherian local DG-ring $(A,\bar{\m})$ is Gorenstein,
if $A$ is a dualizing DG-module over itself,
or, equivalently, if $A$ has finite injective dimension over itself.
This implies that $A$ has bounded cohomology.
If $A$ is a noetherian DG-ring, 
we will say that it is Gorenstein if for all $\bar{\p} \in \opn{Spec}(\mrm{H}^0(A))$,
the local DG-ring $A_{\bar{\p}}$ is Gorenstein.
It is sufficient to check this for all maximal ideals in $\mrm{H}^0(A)$.
Our definition in the non-local case is slightly more general than the commutative version of the definition in \cite{FI} 
as it also allows Gorenstein DG-rings $A$ with $\dim(\mrm{H}^0(A))$ being infinite.

\subsection{Derived completion}

Given a commutative DG-ring $A$,
and a finitely generated ideal $\bar{\a} \subseteq \mrm{H}^0(A)$,
we introduced in \cite[Section 4]{ShComp} the derived $\bar{\a}$-adic completion of $A$ with respect to $\bar{\a}$.
This is a commutative non-positive DG-ring,
denoted by $\mrm{L}\Lambda(A,\bar{\a})$.
If $A$ is noetherian then $\mrm{L}\Lambda(A,\bar{\a})$ is also noetherian,
and if $A$ has bounded cohomology then $\mrm{L}\Lambda(A,\bar{\a})$ also has bounded cohomology.
Moreover, if $(A,\bar{\m})$ is a noetherian local DG-ring,
then $\mrm{L}\Lambda(A,\bar{\m})$ is also a noetherian local DG-ring.
A local DG-ring $(A,\bar{\m})$ is called derived $\bar{\m}$-adically complete if
$A \cong \mrm{L}\Lambda(A,\bar{\m})$.
The DG-ring $\mrm{L}\Lambda(A,\bar{\m})$ is always derived $\bar{\m}$-adically complete.

\subsection{Injectives over DG-rings}

Given a commutative noetherian DG-ring $A$,
by \cite{MiInj,ShInj},
there exist a full subcategory $\opn{Inj}(A) \subseteq \cat{D}^{+}(A)$,
which is a DG version of the category of injective modules over a ring.
The Matlis classification of injectives holds in this setting,
so up to isomorphism, elements of $\opn{Inj}(A)$ are in bijection with $\opn{Spec}(\mrm{H}^0(A))$.
We will denote the element of $\opn{Inj}(A)$ which corresponds to a given $\bar{\p} \in \opn{Spec}(\mrm{H}^0(A))$ by $E(A,\bar{\p})$.

\subsection{Regular sequences and associated primes over commutative DG-rings}\label{sec:regSeq}

Following \cite[Section 2.2]{Mi}, \cite[Section 5]{ShCM},
as well as the earlier \cite{ChI,ChII},
given a commutative noetherian local DG-ring $(A,\bar{\m})$,
and given $M \in \cat{D}^{+}(A)$,
we say that an element $\bar{x} \in \bar{\m}$ is $M$-regular if it is $\mrm{H}^{\inf(M)}(M)$-regular.
That is, if the multiplication map 
\[
\bar{x}: \mrm{H}^{\inf(M)}(M) \to \mrm{H}^{\inf(M)}(M)
\]
is injective.
Inductively, we will say that a finite sequence $\bar{a}_1,\dots,\bar{a}_n \in \bar{\m}$ is $M$-regular
if $\bar{a}_1$ is $M$-regular,
and the sequence $\bar{a}_2,\dots,\bar{a}_n$ is $K(M;\bar{a}_1)$-regular.
Here, $K(M;\bar{a}_1)$ is the cone of the map $\bar{a}_1:M \to M$ in $\cat{D}(A)$,
or, equivalently, the Koszul DG-module of $M$ with respect to $\bar{a}_1$,
defined in Definition \ref{def:KosDG} below.
The maximal length of an $M$-regular sequence in $\bar{\m}$ is a well defined non-negative integer,
denoted by $\opn{seq.depth}_A(M)$,
and is called the sequential depth of $M$ over $A$.
The depth of $M$ is defined to be the number $\opn{depth}_A(M) := \inf(\mrm{R}\opn{Hom}_A(\mrm{H}^0(A)/\bar{\m},M))$,
and the local cohomology Krull dimension of $M$ is the number
\[
\lcdim(M) := \sup_{n \in \mathbb{Z}} \left\{\dim(\mrm{H}^{n}(M)) + n\right\}.
\]

The set of associated prime ideals of $M$, denoted by $\opn{Ass}_A(M)$,
is by definition the set $\{\bar{\p} \in \opn{Supp}_A(M) \mid \depth_{A_{\bar{\p}}}(M_{\bar{\p}}) = \inf(M_{\bar{\p}})\}$,
where we let $\opn{Supp}_A(M) = \{\bar{\p} \in \opn{Spec}(\mrm{H}^0(A)) \mid M_{\bar{\p}} \ncong 0\}$.
In a dual manner, we define the set $W_0^A(M)$
to be the set $\{\bar{\p} \in \opn{Supp}_A(M) \mid \lcdim_{A_{\bar{\p}}}(M_{\bar{\p}}) = \sup(M_{\bar{\p}})\}$.
If $M \in \cat{D}^{\mrm{b}}_{\mrm{f}}(A)$,
then by \cite[Propositions 5.9 and 5.11]{ShCM}, 
the sets $\opn{Ass}_A(M)$ and $W_0^A(M)$ are finite sets.

\subsection{Cohen-Macaulay DG-rings}

Given a commutative noetherian local DG-ring $(A,\bar{\m})$ with bounded cohomology,
by \cite[Corollary 5.5]{ShCM},
there is an inequality $\opn{seq.depth}_A(A) \le \dim(\mrm{H}^0(A))$.
If there is an equality $\opn{seq.depth}_A(A) = \dim(\mrm{H}^0(A))$ then $A$ is called a local-Cohen-Macaulay DG-ring.
The reason for the name local-Cohen-Macaulay is because this definition is not stable under localization.
We say that commutative noetherian DG-ring $A$ with bounded cohomology is Cohen-Macaulay if for all $\bar{\p} \in \opn{Spec}(\mrm{H}^0(A))$,
the local DG-ring $A_{\bar{\p}}$ is local-Cohen-Macaulay. 
All Gorenstein DG-rings are Cohen-Macaulay.

If a noetherian local DG-ring $A$ with bounded cohomology has a dualizing DG-module $R$,
then $A$ is local-Cohen-Macaulay if and if and only if $\amp(A) = \amp(R)$.
Many other equivalent conditions to the local-Cohen-Macaulay condition are given in \cite[Theorem 2]{ShCM}.
If $A$ is local-Cohen-Macaulay and $\opn{Spec}(\mrm{H}^0(A))$ is irreducible,
or if $\opn{Supp}(\mrm{H}^{\inf(A)}(A)) = \opn{Spec}(\mrm{H}^0(A))$,
then $A$ is Cohen-Macaulay.

\section{Koszul DG-modules over commutative DG-rings}\label{sec:Koszul}

The aim of this section is to study the Koszul complex
over commutative DG-rings.
Minamoto began such a study in \cite[Section 2.2]{Mi},
where he mainly focused on the Koszul complex with respect to a single regular element.
We first recall the situation over ordinary commutative rings.

Given a commutative ring $A$,
and an element $a \in A$,
recall that the Koszul complex $K(A;a)$ is the complex
\[
0 \to A \xrightarrow{\cdot a} A \to 0,
\]
concentrated in cohomological degrees $-1,0$.
This is has the structure of a commutative DG-ring by letting $a\cdot b = 0$ if $\deg(a) = \deg(b) = -1$,
and by making $K(A;a)^{-1}$ a free $A$-module of rank $1$ with basis $1$.
The identity map of $A$ induces a natural map of DG-rings $\kappa:A \to K(A;a)$.
Observe that there is an isomorphism of complexes of $A$-modules:
\begin{equation}\label{eqn:dual1}
\opn{Hom}_A(K(A;a),A) \cong K(A;a)[-1].
\end{equation}

Given a sequence $a_1,\dots,a_n \in A$,
the Koszul complex associated to $A$ and $a_1,\dots,a_n$ is defined by:
\[
K(A;a_1,\dots,a_n) := K(A;a_1)\otimes_A \dots \otimes_A K(A;a_n),
\]
where the tensor product is taken in the category of commutative DG-algebras over $A$.
As a complex of $A$-modules, $K(A;a_1,\dots,a_n)$ is a bounded complex of finitely generated free $A$-modules.

As is well known (see for instance \cite[Theorem 16.5]{Mat}),
if $a_1,\dots,a_n$ is an $A$-regular sequence,
then the natural map
\[
K(A;a_1,\dots,a_n) \to \mrm{H}^0\left(K(A;a_1,\dots,a_n)\right) = A/(a_1,\dots,a_n).
\]
is a quasi-isomorphism of DG-rings. 
It follows that in this case, $K(A;a_1,\dots,a_n)$ is a semi-free commutative DG-algebra resolution of $A/(a_1,\dots,a_n)$ over $A$.

It follows from (\ref{eqn:dual1}) that there is an isomorphism
\begin{equation}\label{eqn:dualn}
\opn{Hom}_A(K(A;a_1,\dots,a_n),A) \cong K(A;a_1,\dots,a_n)[-n]
\end{equation}
of complexes of $A$-modules.

If $M$ is an $A$-module,
we let
\[
K(M;a_1,\dots,a_n) := K(A;a_1,\dots,a_n)\otimes_A M.
\]
This is a DG-module over $K(A;a_1,\dots,a_n)$.

If $A \to B$ is a map of commutative rings,
and if $f(a_i) = b_i$,
we obtain an isomorphism of DG-rings:
\begin{equation}\label{eqn:base-change}
K(A;a_1,\dots,a_n)\otimes_A B \cong K(B;b_1,\dots,b_n).
\end{equation}

We now define the notions of Koszul DG-modules and Koszul DG-rings over a commutative DG-ring,
following \cite[Section 2.2]{Mi}.

\begin{dfn}\label{def:KosDG}
Let $A$ be a commutative DG-ring,
and let $\bar{a}_1,\dots,\bar{a}_n \in \mrm{H}^0(A)$.
\begin{enumerate}
\item We define the Koszul DG-ring $K(A;\bar{a}_1,\dots,\bar{a}_n)$ associated to $A$ and $\bar{a}_1,\dots,\bar{a}_n$ as follows:
for each $1 \le i \le n$, choose some $a_i \in A^0$, 
such that $\pi^0_A(a_i) = \bar{a}_i$.
Give $A$ a DG-algebra structure over $\mathbb{Z}[x_1,\dots,x_n]$ by letting $x_i \mapsto a_i$,
and finally, 
define
\[
K(A;\bar{a}_1,\dots,\bar{a}_n) := A\otimes^{\mrm{L}}_{\mathbb{Z}[x_1,\dots,x_n]} \mathbb{Z},
\]
where the derived tensor product is calculated in the category of commutative DG-rings.
\item Given a DG-module $M \in \cat{D}(A)$,
we define the Koszul DG-module associated to $M$ and $\bar{a}_1,\dots,\bar{a}_n$ by
\[
K(M;\bar{a}_1,\dots,\bar{a}_n) := M\otimes^{\mrm{L}}_{\mathbb{Z}[x_1,\dots,x_n]} \mathbb{Z} \in \cat{D}(K(A;\bar{a}_1,\dots,\bar{a}_n)).
\]
\end{enumerate}
\end{dfn}

In \cite[Section 2.2]{Mi}, what we denote here by $K(A;\bar{x})$ was denoted by $A//\bar{x}$.
We will soon show that this definition is independent of the chosen lifts of $\bar{a}_1,\dots,\bar{a}_n$,
but to do this, we first need the following:

\begin{prop}\label{prop:KosDGcon}
Let $A$ be a commutative DG-ring,
and let $\bar{a}_1,\dots,\bar{a}_n,\bar{b}_1,\dots,\bar{b}_m \in \mrm{H}^0(A)$.
After choosing some lifts $a_1,\dots,a_n,b_1,\dots,b_m \in A^0$,
there is an isomorphism of DG-rings
\[
K(A;\bar{a}_1,\dots,\bar{a}_n) \otimes^{\mrm{L}}_A K(A;\bar{b}_1,\dots,\bar{b}_m) \cong K(A;\bar{a}_1,\dots,\bar{a}_n,\bar{b}_1,\dots,\bar{b}_m).
\]
\end{prop}
\begin{proof}
By definition, we have that
\[
K(A;\bar{a}_1,\dots,\bar{a}_n) \otimes^{\mrm{L}}_A K(A;\bar{b}_1,\dots,\bar{b}_m) =
\left( A\otimes^{\mrm{L}}_{\mathbb{Z}[x_1,\dots,x_n]} \mathbb{Z} \right) \otimes^{\mrm{L}}_A 
\left( A\otimes^{\mrm{L}}_{\mathbb{Z}[y_1,\dots,y_m]} \mathbb{Z} \right). 
\]
Give $A$ the structure of a DG-algebra over $\mathbb{Z}[x_1,\dots,x_n,y_1,\dots,y_m]$ by letting $x_i \mapsto a_i$, $y_i \mapsto b_i$.
Since the maps $\mathbb{Z}[x_1,\dots,x_n] \to A$ and $\mathbb{Z}[y_1,\dots,y_m] \to A$ factor as
\[
\mathbb{Z}[x_1,\dots,x_n] \to \mathbb{Z}[x_1,\dots,x_n,y_1,\dots,y_m] \to A
\]
and
\[
\mathbb{Z}[y_1,\dots,y_m] \to \mathbb{Z}[x_1,\dots,x_n,y_1,\dots,y_m] \to A,
\]
and since there are isomorphisms of DG-algebras
\[
\mathbb{Z} \cong K(\mathbb{Z}[x_1,\dots,x_n];x_1,\dots,x_n)
\]
over $\mathbb{Z}[x_1,\dots,x_n]$
and
\[
\mathbb{Z} \cong K(\mathbb{Z}[y_1,\dots,y_m];y_1,\dots,y_m)
\]
over $\mathbb{Z}[y_1,\dots,y_m]$,
by (\ref{eqn:base-change}), 
we know that there are isomorphisms
\[
\left( A\otimes^{\mrm{L}}_{\mathbb{Z}[x_1,\dots,x_n]} \mathbb{Z} \right) \cong
\left( A\otimes^{\mrm{L}}_{\mathbb{Z}[x_1,\dots,x_n,y_1,\dots,y_m]} K(\mathbb{Z}[x_1,\dots,x_n,y_1,\dots,y_m];x_1,\dots,x_n) \right)
\]
and
\[
\left( A\otimes^{\mrm{L}}_{\mathbb{Z}[y_1,\dots,y_m]} \mathbb{Z} \right) \cong
\left( A\otimes^{\mrm{L}}_{\mathbb{Z}[x_1,\dots,x_n,y_1,\dots,y_m]} K(\mathbb{Z}[x_1,\dots,x_n,y_1,\dots,y_m];y_1,\dots,y_m) \right).
\]
Combining these with the isomorphism
\begin{gather*}
K(\mathbb{Z}[x_1,\dots,x_n,y_1,\dots,y_m];x_1,\dots,x_n) \otimes^{\mrm{L}}_{\mathbb{Z}[x_1,\dots,x_n,y_1,\dots,y_m]}\\
K(\mathbb{Z}[x_1,\dots,x_n,y_1,\dots,y_m];y_1,\dots,y_m) \cong\\
K(\mathbb{Z}[x_1,\dots,x_n,y_1,\dots,y_m];x_1,\dots,x_n,y_1,\dots,y_m)
\end{gather*}
and with associativity of the derived tensor product, we obtain the required result.
\end{proof}

\begin{prop}\label{prop:kosInd}
Up to an isomorphism in $\ho(\cdg)$,
Definition \ref{def:KosDG}(1) is independent of the chosen $a_1,\dots,a_n$.
\end{prop}
\begin{proof}
Suppose $a_1',\dots,a_n' \in A^0$ is another choice of elements which satisfy $\pi^0_A(a_i') = \bar{a}_i$ for all $1 \le i \le n$.
Let us temporary denote the Koszul DG-ring associated to $A$ and a lift $a_1,\dots,a_n$ by $K(a_1,\dots,a_n)$.
We will show the result by induction on $n$.
If $n=1$, then it is shown in \cite[Lemma 2.8]{Mi} that there is a commutative DG-ring $T$,
and quasi-isomorphisms $K(a_1) \to T$ and $K(a_1') \to T$,
which gives the required isomorphism $K(a_1) \cong K(a_1')$ in $\ho(\cdg)$.
Suppose that the claim is true for all Koszul DG-rings with respect to sequences of length less than $n$.
Then by Proposition \ref{prop:KosDGcon},
we have a sequence of isomorphisms in $\ho(\cdg)$:
\begin{gather*}
K(a_1,\dots,a_n) \cong K(a_1,\dots,a_{n-1}) \otimes^{\mrm{L}}_A K(a_n) \cong\\
K(a_1',\dots,a_{n-1}') \otimes^{\mrm{L}}_A K(a_n') \cong K(a_1',\dots,a_n').
\end{gather*}
\end{proof}

\begin{rem}
Given a commutative DG-ring $A$, and $\bar{a}_1,\dots,\bar{a}_n \in \mrm{H}^0(A)$,
the natural surjection $\mathbb{Z}[x_1,\dots,x_n] \surj \mathbb{Z}$ induces a map of DG-rings
\[
A =  A\otimes^{\mrm{L}}_{\mathbb{Z}[x_1,\dots,x_n]} \mathbb{Z}[x_1,\dots,x_n] \to  A\otimes^{\mrm{L}}_{\mathbb{Z}[x_1,\dots,x_n]} \mathbb{Z} = K(A;\bar{a}_1,\dots,\bar{a}_n).
\]
We denote this map by $\kappa:A \to K(A;\bar{a}_1,\dots,\bar{a}_n)$,
and remark that these maps commute with the isomorphisms in Proposition \ref{prop:kosInd}.
\end{rem}

\begin{prop}
Up to isomorphism in $\cat{D}(K(A;\bar{a}_1,\dots,\bar{a}_n)$,
Definition \ref{def:KosDG}(2) is independent of the chosen $a_1,\dots,a_n$.
\end{prop}
\begin{proof}
This follows from Proposition \ref{prop:kosInd} and the isomorphism
\[
K(M;\bar{a}_1,\dots,\bar{a}_n) \cong K(A;\bar{a}_1,\dots,\bar{a}_n) \otimes^{\mrm{L}}_A M.
\]
\end{proof}

\begin{prop}\label{prop:KoszulChange}
Let $f:A \to B$ be a map of commutative DG-rings,
let $\bar{a}_1,\dots,\bar{a}_n \in \mrm{H}^0(A)$,
and for any $1 \le i \le n$,
let $\bar{b}_i = \mrm{H}^0(f)(\bar{a}_i)$.
Then there is an isomorphism
\[
K(A;\bar{a}_1,\dots,\bar{a}_n) \otimes^{\mrm{L}}_A B \cong
K(B;\bar{b}_1,\dots,\bar{b}_n)
\]
in $\ho(\cdg)$.
\end{prop}
\begin{proof}
Let $a_1,\dots,a_n \in A^0$ be such that $\pi^0_A(a_i) = \bar{a}_i$,
and let $b_i = f(a_i)$ for all $1 \le i \le n$.
Making $A,B$ to DG-algebras over $\mathbb{Z}[x_1,\dots,x_n]$ DG-algebra by setting $x_i \mapsto a_i$ and $x_i \mapsto b_i$,
we see that the map $f:A \to B$ is $\mathbb{Z}[x_1,\dots,x_n]$-linear, 
so we get using these $\mathbb{Z}[x_1,\dots,x_n]$-structures:
\begin{gather*}
K(A;\bar{a}_1,\dots,\bar{a}_n) \otimes^{\mrm{L}}_A B \cong
(A\otimes^{\mrm{L}}_{\mathbb{Z}[x_1,\dots,x_n]} \mathbb{Z}) \otimes^{\mrm{L}}_A B \cong\\
B\otimes^{\mrm{L}}_{\mathbb{Z}[x_1,\dots,x_n]} \mathbb{Z} \cong
K(B;\bar{b}_1,\dots,\bar{b}_n).
\end{gather*}
\end{proof}

\begin{prop}\label{prop:compact}
Given a commutative DG-ring $A$, and $\bar{a}_1,\dots,\bar{a}_n \in \mrm{H}^0(A)$,
considered as an object of $\cat{D}(A)$,
the DG-module $K(A;\bar{a}_1,\dots,\bar{a}_n)$ is a compact object.
\end{prop}
\begin{proof}
Applying Proposition \ref{prop:KoszulChange} to the map $\pi_A:A \to \mrm{H}^0(A)$,
we have an isomorphism
\[
K(A;\bar{a}_1,\dots,\bar{a}_n) \otimes^{\mrm{L}}_A \mrm{H}^0(A) \cong
K(\mrm{H}^0(A);\bar{a}_1,\dots,\bar{a}_n).
\]
Since $K(\mrm{H}^0(A);\bar{a}_1,\dots,\bar{a}_n)$,
being a bounded complex of finitely generated free $\mrm{H}^0(A)$-modules,
is a compact object of $\cat{D}(\mrm{H}^0(A))$,
it follows form \cite[Theorem 5.11]{Ye1} and 
\cite[Theorem 5.20]{Ye1}
that $K(A;\bar{a}_1,\dots,\bar{a}_n)$ is a compact object of $\cat{D}(A)$.
\end{proof}

The Koszul complex commutes with localization:

\begin{prop}\label{prop:kosLocal}
Let $A$ be a commutative DG-ring, and let $\bar{a}_1,\dots,\bar{a}_n \in \mrm{H}^0(A)$.
Consider the canonical surjection $\tau:\mrm{H}^0(A) \to \mrm{H}^0(A)/(\bar{a}_1,\dots,\bar{a}_n)$.
Given 
\[
\bar{\p} \in \opn{Spec}(\mrm{H}^0(A)/(\bar{a}_1,\dots,\bar{a}_n)),
\]
letting $\bar{\q} = \tau^{-1}(\bar{\p})$,
there is an isomorphism of DG-rings
\[
K(A;\bar{a}_1,\dots,\bar{a}_n)_{\bar{\p}} \cong K(A_{\bar{\q}};\bar{a}_1/1,\dots,\bar{a}_n/1).
\]
\end{prop}
\begin{proof}
Since the degree zero part of the Koszul complex satisfies
\[
\left(K(\mathbb{Z}[x_1,\dots,x_n];x_1,\dots,x_n])\right)^0 = \mathbb{Z}[x_1,\dots,x_n],
\]
it follows that
\[
\left(A\otimes^{\mrm{L}}_{\mathbb{Z}[x_1,\dots,x_n]} \mathbb{Z}\right)^0 = A^0.
\]
Hence, there is an equality
\[
\left(\pi^0_{A\otimes^{\mrm{L}}_{\mathbb{Z}[x_1,\dots,x_n]} \mathbb{Z}}\right)^{-1}(\bar{\p}) = (\pi^0_A)^{-1}(\bar{\q}).
\]
Let us denote this prime ideal of $A^0$ by $\p$.
By the definition of localization we obtain the following sequence of isomorphisms of DG-rings:
\begin{gather*}
K(A;\bar{a}_1,\dots,\bar{a}_n)_{\bar{\p}} = K(A;\bar{a}_1,\dots,\bar{a}_n) \otimes^{\mrm{L}}_{A^0} A^0_{\p} \cong\\
\left(A\otimes^{\mrm{L}}_{\mathbb{Z}[x_1,\dots,x_n]} \mathbb{Z}\right) \otimes^{\mrm{L}}_{A^0} A^0_{\p} \cong\\
A_{\bar{\q}}\otimes^{\mrm{L}}_{\mathbb{Z}[x_1,\dots,x_n]} \mathbb{Z} = K(A_{\bar{\q}};\bar{a}_1/1,\dots,\bar{a}_n/1).
\end{gather*}
\end{proof}

\begin{prop}\label{prop:dualKosz}
Let $A$ be a commutative DG-ring, and let $\bar{a}_1,\dots,\bar{a}_n \in \mrm{H}^0(A)$.
Then there is an isomorphism
\[
\mrm{R}\opn{Hom}_A(K(A;\bar{a}_1,\dots,\bar{a}_n),A) \cong K(A;\bar{a}_1,\dots,\bar{a}_n)[-n]
\]
in $\cat{D}(K(A;\bar{a}_1,\dots,\bar{a}_n))$.
\end{prop}
\begin{proof}
Let $K = K(\mathbb{Z}[x_1,\dots,x_n];x_1,\dots,x_n)$.
Note that $K \cong \mathbb{Z}$, and that $K$ is a bounded complex of finitely generated free $\mathbb{Z}[x_1,\dots,x_n]$-modules.
Then $K(A;\bar{a}_1,\dots,\bar{a}_n) = K\otimes_{\mathbb{Z}[x_1,\dots,x_n]} A$ is K-projective over $A$,
and we have the following sequence of isomorphisms in $\cat{D}(K(A;\bar{a}_1,\dots,\bar{a}_n))$:
\begin{gather*}
\mrm{R}\opn{Hom}_A(K(A;\bar{a}_1,\dots,\bar{a}_n),A) =\\
\opn{Hom}_A(K\otimes_{\mathbb{Z}[x_1,\dots,x_n]} A,A) \cong \\
\opn{Hom}_{\mathbb{Z}[x_1,\dots,x_n]}(K,A) \cong\\
\opn{Hom}_{\mathbb{Z}[x_1,\dots,x_n]}(K,\mathbb{Z}[x_1,\dots,x_n]) \otimes_{\mathbb{Z}[x_1,\dots,x_n]} A \cong^{(\diamond)}\\
K[-n] \otimes_{\mathbb{Z}[x_1,\dots,x_n]} A = K(A;\bar{a}_1,\dots,\bar{a}_n)[-n].
\end{gather*}
where the isomorphism $(\diamond)$ is by (\ref{eqn:dualn}).
\end{proof}

The last result of this section shows that Koszul complexes of a noetherian local DG-ring commute with derived adic completion of DG-rings.
Before we prove this result, we need the following lemma.

\begin{lem}\label{lem:injSurj}
Let $f:(A,\bar{\m})\to (B,\bar{\n})$ be a map of DG-rings between commutative noetherian local DG-rings,
such that $\mrm{H}^0(f):\mrm{H}^0(A) \to \mrm{H}^0(B)$ is a surjective local homomorphism.
Then there is an isomorphism
\[
\mrm{R}\opn{Hom}_A(B,E(A,\bar{\m})) \cong E(B,\bar{\n})
\]
in $\cat{D}(B)$.
\end{lem}
\begin{proof}
Let $J = \mrm{R}\opn{Hom}_A(B,E(A,\bar{\m}))$.
By \cite[Proposition 5.6]{ShInj},
we have that $J \in \opn{Inj}(B)$.
Hence, by \cite[Theorem 5.7]{ShInj},
it is enough to show that $\mrm{H}^0(J) \cong \mrm{H}^0(E(B,\bar{\n}))$.
By definition, $\mrm{H}^0(E(B,\bar{\n}))$ is the injective hull over $\mrm{H}^0(B)$ of the residue field $\mrm{H}^0(B)/\bar{\n}$.
On the other hand, by \cite[Theorem 4.10]{ShInj},
we have that
\[
\mrm{H}^0(J) = \mrm{H}^0\left(\mrm{R}\opn{Hom}_A(B,E(A,\bar{\m}))\right) \cong \opn{Hom}_{\mrm{H}^0(A)}\left(\mrm{H}^0(B),\mrm{H}^0(E(A,\bar{\m}))\right).
\]
According to \cite[tag 08Z2]{SP}, since $\mrm{H}^0(A) \to \mrm{H}^0(B)$ is a surjection of local rings,
we have that
\[
\opn{Hom}_{\mrm{H}^0(A)}\left(\mrm{H}^0(B),\mrm{H}^0(E(A,\bar{\m}))\right)
\]
is isomorphic to the injective hull over $\mrm{H}^0(B)$ of the residue field $\mrm{H}^0(B)/\bar{\n}$, as claimed.
\end{proof}

Before proving that derived completion commutes with the Koszul complex,
recall, following \cite[Theorem 7.11]{Mat},
that if $A \to B$ is a map of commutative rings,
and if $M,N$ are $A$-modules, 
there is a natural map
\begin{equation}\label{eqn:nat-end}
\opn{Hom}_A(M,N)\otimes_A B \to \opn{Hom}_B(M\otimes_A B,N\otimes_A B),
\end{equation}
given by $f\otimes b \mapsto b\cdot (f\otimes 1_B)$.
Further note that in case $M=N$,
so that both sides of (\ref{eqn:nat-end}) are rings,
then the above map is a ring homomorphism.
One may factor (\ref{eqn:nat-end}) as
\begin{equation}\label{eqn:factor-nat-end}
\opn{Hom}_A(M,N)\otimes_A B \to \opn{Hom}_A(M,N\otimes_A B) \to \opn{Hom}_B(M\otimes_A B,N\otimes_A B),
\end{equation}
where the first map is the tensor-evaluation morphism,
and the second map is the hom-tensor adjunction.
This entire discussion may be generalized to DG-rings as follows.
Assume that $A \to B$ is a map of commutative DG-rings,
and let $M \in \cat{D}(A)$.
Let $P \to M$ be a K-projective resolution,
and let $A \to \tilde{B} \cong B$ be a K-flat commutative DG-algebra resolution of $B$ over $A$
(such a resolution exists by \cite[Theorem 3.21]{YeSq}).
It follows that $P\otimes_A \tilde{B}$ is a K-projective resolution of $M\otimes^{\mrm{L}}_A \tilde{B}$ over $\tilde{B}$,
so we obtain natural maps
\begin{gather*}
\mrm{R}\opn{Hom}_A(M,M)\otimes^{\mrm{L}}_A B \cong
\opn{Hom}_A(P,P)\otimes_A \tilde{B} \to^{(\diamond)}\\
\opn{Hom}_A(P,P\otimes_A \tilde{B}) \cong
\opn{Hom}_{\tilde{B}}(M\otimes_A \tilde{B},N\otimes_A \tilde{B}) \cong\\
\mrm{R}\opn{Hom}_B(M\otimes^{\mrm{L}}_A B,M\otimes^{\mrm{L}}_A B).
\end{gather*}

As before, we see that this composition is a map in the homotopy category of DG-rings,
and that it is an isomorphism if and only if the tensor-evaluation morphism,
denoted above by $(\diamond)$, is an isomorphism.
See \cite[Section 12.9]{YeBook} for a detailed discussion of the tensor-evaluation morphism over DG-rings.

We are now ready to prove that the Koszul complex commutes with derived completion.
We remark that we need the following rather difficult proof,
because our noetherian assumption is only on the level of cohomology.
One can give a much simpler proof,
using the base change property of the Koszul complex,
if one assumes that the ring $A^0$ is noetherian,
but as our proof demonstrates, 
the statement remains true without this assumption.

\begin{thm}\label{thm:koszulCompletion}
Let $(A,\bar{\m})$ be a commutative noetherian local DG-ring,
and suppose that $\bar{a}_1,\dots,\bar{a}_n \in \bar{\m}$.
Denote by $\bar{\n}$ the image of $\bar{\m}$ in the quotient ring $\mrm{H}^0(A)/(\bar{a}_1,\dots,\bar{a}_n)$,
Denote the image of $\bar{a}_1,\dots,\bar{a}_n$ under the completion map $\mrm{H}^0(A) \to \Lambda_{\bar{\m}}(\mrm{H}^0(A))$
by $\widehat{\bar{a}}_1,\dots,\widehat{\bar{a}}_n$.
Then there is an isomorphism
\[
K( \mrm{L}\Lambda(A,\bar{\m}); \widehat{\bar{a}}_1,\dots,\widehat{\bar{a}}_n) \cong \mrm{L}\Lambda( K(A;\bar{a}_1,\dots,\bar{a}_n),\bar{\n})
\]
in $\ho(\cdg)$.
\end{thm}
\begin{proof}
Let us set $B:= K(A;\bar{a}_1,\dots,\bar{a}_n)$.
Since $\mrm{H}^{0}(\kappa):\mrm{H}^0(A) \to \mrm{H}^0(B)$ is surjective, 
it follows from Lemma \ref{lem:injSurj} that
\[
\mrm{R}\opn{Hom}_A(B,E(A,\bar{\m})) \cong E(B,\bar{\n}).
\]
According to \cite[Theorem 7.22]{ShInj},
there is an isomorphism
\[
\mrm{R}\opn{Hom}_B(E(B,\bar{\n}),E(B,\bar{\n})) \cong \mrm{L}\Lambda(B,\bar{\n})
\]
in $\ho(\cdg)$.
By Proposition \ref{prop:compact},
considered as an object of $\cat{D}(A)$,
we know that $B$ is compact.
Hence, by \cite[Theorem 14.1.22]{YeBook}, the map
\[
\mrm{R}\opn{Hom}_A(B,A) \otimes^{\mrm{L}}_A E(A,\bar{\m}) \to \mrm{R}\opn{Hom}_A(B,E(A,\bar{\m}))
\]
is an isomorphism in $\cat{D}(B)$.
By Proposition \ref{prop:dualKosz},
we know that $\mrm{R}\opn{Hom}_A(B,A) \cong B[-n]$.
Those isomorphisms imply that
\begin{gather*}
\mrm{L}\Lambda(B,\bar{\n}) \cong \mrm{R}\opn{Hom}_B(B[-n]\otimes^{\mrm{L}}_A E(A,\bar{\m}),B[-n]\otimes^{\mrm{L}}_A E(A,\bar{\m})) \cong\\
\mrm{R}\opn{Hom}_B(B\otimes^{\mrm{L}}_A E(A,\bar{\m}),B\otimes^{\mrm{L}}_A E(A,\bar{\m}))
\end{gather*}
in $\ho(\cdg)$.
By the discussion preceding this theorem,
we know that there is a map in $\ho(\cdg)$:
\[
\mrm{R}\opn{Hom}_A(E(A,\bar{\m}),E(A,\bar{\m})) \otimes^{\mrm{L}}_A B \to
\mrm{R}\opn{Hom}_B(B\otimes^{\mrm{L}}_A E(A,\bar{\m}),B\otimes^{\mrm{L}}_A E(A,\bar{\m})),
\]
and that it is an isomorphism if and only if the map
\begin{equation}\label{eqn:neededTensorEval}
\mrm{R}\opn{Hom}_A(E(A,\bar{\m}),E(A,\bar{\m})) \otimes^{\mrm{L}}_A B \to \mrm{R}\opn{Hom}_A(E(A,\bar{\m}),E(A,\bar{\m}) \otimes^{\mrm{L}}_A B )
\end{equation}
is an isomorphism.
Assuming for a moment that this is the case,
we deduce that there is an isomorphism
\[
\mrm{L}\Lambda(B,\bar{\n}) \cong \mrm{R}\opn{Hom}_A(E(A,\bar{\m}),E(A,\bar{\m})) \otimes^{\mrm{L}}_A B
\]
in $\ho(\cdg)$.
Invoking \cite[Theorem 7.22]{ShInj} again, we see that
\[
\mrm{R}\opn{Hom}_A(E(A,\bar{\m}),E(A,\bar{\m})) \otimes^{\mrm{L}}_A B \cong \mrm{L}\Lambda(A,\bar{\m}) \otimes^{\mrm{L}}_A B,
\]
and by Proposition \ref{prop:KoszulChange},
the latter is isomorphic to
\[
K( \mrm{L}\Lambda(A,\bar{\m}); \widehat{\bar{a}}_1,\dots,\widehat{\bar{a}}_n),
\]
as claimed.
The conclusion of the theorem will now follow from the fact that (\ref{eqn:neededTensorEval}) is an isomorphism,
and this in turn follows from Proposition \ref{prop:Tensor-eval-with-compact} below.
\end{proof}

\begin{prop}\label{prop:Tensor-eval-with-compact}
Let $A$ be a commutative noetherian DG-ring,
let $M,N \in \cat{D}(A)$,
and let $K \in \cat{D}(A)$ be a compact object.
Then the natural map
\[
\mrm{R}\opn{Hom}_A(M,N) \otimes^{\mrm{L}}_A K \to \mrm{R}\opn{Hom}_A(M,N \otimes^{\mrm{L}}_A K)
\]
is an isomorphism in $\cat{D}(A)$.
\end{prop}
\begin{proof}
Fixing $M,N$,
it is shown in \cite[Theorem 12.9.10]{YeBook} that there is a functorial morphism
between triangulated functors
\[
\zeta_K: \mrm{R}\opn{Hom}_A(M,N) \otimes^{\mrm{L}}_A K \to \mrm{R}\opn{Hom}_A(M,N \otimes^{\mrm{L}}_A K).
\]
By \cite[Theorem 14.1.22]{YeBook},
the fact that $K$ is compact is equivalent to the fact that $K$ belongs to the saturated full triangulated subcategory of $\cat{D}(A)$ generated by $A$.
Then, by \cite[Proposition 5.3.22]{YeBook},
the fact that $\zeta_A$ is an isomorphism implies that $\zeta_K$ is an isomorphism for any such $K$.
\end{proof}

\section{Depth and height in Cohen-Macaulay DG-rings}\label{sec:dim}

The aim of this section is to study the following definition over Cohen-Macaulay DG-rings:

\begin{dfn}
Let $(A,\bar{\m})$ be a noetherian local DG-ring,
let $\bar{I} \subseteq \mrm{H}^0(A)$ be a proper ideal,
and let $M \in \cat{D}^{+}(A)$.

\begin{enumerate}
\item The $\bar{I}$-depth of $M$ is the number
\[
\opn{depth}_A(\bar{I},M) := \inf\left(\mrm{R}\opn{Hom}_A(\mrm{H}^0(A)/\bar{I},M)\right).
\]
\item The sequential $\bar{I}$-depth of $M$,
denoted by $\opn{seq.depth}_A(\bar{I},M)$,
is defined to be the maximal length of an $M$-regular sequence contained in $\bar{I}$.
\end{enumerate}
When $\bar{I} = \bar{\m}$, the maximal ideal of $\mrm{H}^0(A)$,
we write $\depth_A(M) := \depth_A(\bar{\m},M)$,
and $\opn{seq.depth}_A(M) := \opn{seq.depth}_A(\bar{\m},M)$.
\end{dfn}

Note that a-priori it is not clear why the sequential $\bar{I}$-depth of $M$ is finite.
The next result connects these two numbers, and in particular establishes the finiteness of the sequential $\bar{I}$-depth,
and shows that any two maximal $M$-regular sequences contained in $\bar{I}$ have the same length. 
In \cite[Theorem 2.15]{Mi} a similar result was proved in the case where $\bar{I} = \bar{\m}$.
 
\begin{prop}\label{prop:depth-formula}
Let $(A,\bar{\m})$ be a noetherian local DG-ring,
and let $0\ncong M \in \cat{D}^{+}_{\mrm{f}}(A)$.
Then for any proper ideal $\bar{I} \subseteq \mrm{H}^0(A)$,
there is an equality
\[
\opn{depth}_A(\bar{I},M) = \opn{seq.depth}_A(\bar{I},M) + \inf(M).
\]
\end{prop}
\begin{proof}
Let $\bar{x}_1,\dots,\bar{x}_n \in \bar{I}$ be any maximal $M$-regular sequence.
By that we mean that it cannot be extended to a longer $M$-regular sequence in $\bar{I}$, 
but it is possible that there might be a maximal $M$-regular sequence of longer length contained in $\bar{I}$.
Let $N = K(M;\bar{x}_1,\dots,\bar{x}_n)$.
Since $\bar{x}_1,\dots,\bar{x}_n$ is $M$-regular, 
it follows from \cite[Lemma 2.13]{Mi} that $\inf(N) = \inf(M)$.
The fact that the sequence $\bar{x}_1,\dots,\bar{x}_n$ cannot be extended to a longer $M$-regular sequence in $\bar{I}$
means that any element $\bar{x}\in \bar{I}$ is not $\mrm{H}^{\inf(N)}(N)$-regular.
This implies that $\opn{seq.depth}_{\mrm{H}^0(A)}(\bar{I},\mrm{H}^{\inf(N)}(N)) = 0$. 
By \cite[Theorem 16.6]{Mat}, this is equivalent to
\[
\opn{Hom}_{\mrm{H}^0(A)}(\mrm{H}^0(A)/\bar{I},\mrm{H}^{\inf(N)}(N)) \ne 0.
\]
Letting $B = K(A;\bar{x}_1,\dots,\bar{x}_n)$,
and denoting by $\bar{J}$ the image of $\bar{I}$ in $\mrm{H}^0(B)$, 
since $\bar{x}_1,\dots,\bar{x}_n \in \bar{I}$,
and since we have $\mrm{H}^0(B) = \mrm{H}^0(A)/(\bar{x}_1,\dots,\bar{x}_n)$,
the above is equivalent to
\[
\opn{Hom}_{\mrm{H}^0(B)}(\mrm{H}^0(B)/\bar{J},\mrm{H}^{\inf(N)}(N)) \ne 0.
\]
Since we have
\[
\mrm{R}\opn{Hom}_B(\mrm{H}^0(B)/\bar{J},N) \cong \mrm{R}\opn{Hom}_{\mrm{H}^0(B)}(\mrm{H}^0(B)/\bar{J},\mrm{R}\opn{Hom}_B(\mrm{H}^0(B),N)),
\]
and since by \cite[Proposition 3.3]{ShInj} we have that
\[
\inf\left(\mrm{R}\opn{Hom}_B(\mrm{H}^0(B),N)\right) = \inf(N),
\]
and 
\[
\mrm{H}^{\inf(N)}\left(\mrm{R}\opn{Hom}_B(\mrm{H}^0(B),N)\right) = \mrm{H}^{\inf(N)}(N),
\]
we see that
\begin{gather*}
\inf(\mrm{R}\opn{Hom}_B(\mrm{H}^0(B)/\bar{J},N)) =\\
\inf\left(\mrm{R}\opn{Hom}_{\mrm{H}^0(B)}(\mrm{H}^0(B)/\bar{J},\mrm{R}\opn{Hom}_B(\mrm{H}^0(B),N))\right) =
\inf(N).
\end{gather*}

This shows that $\opn{depth}_B(\bar{J},N) = \inf(N) = \inf(M)$.
On the other hand, by a repeated use of \cite[Lemma 2.9]{Mi},
we have that
\begin{gather*}
\mrm{R}\opn{Hom}_A(\mrm{H}^0(A)/\bar{I},M) \cong \mrm{R}\opn{Hom}_{K(A;\bar{x}_1)}(\mrm{H}^0(A)/\bar{I},K(M;\bar{x}_1)[-1]) \cong\\
\mrm{R}\opn{Hom}_{K(A;\bar{x}_1,\bar{x}_2)}(\mrm{H}^0(A)/\bar{I},K(M;\bar{x}_1,\bar{x}_2)[-2]) \cong \dots \cong
\mrm{R}\opn{Hom}_B(\mrm{H}^0(B)/\bar{J},N[-n]).
\end{gather*}
This implies that $\opn{depth}_B(\bar{J},N) = \opn{depth}_A(\bar{I},M) - n$,
and combined with the above calculation of $\opn{depth}_B(\bar{J},N) = \inf(M)$,
it shows that
\[
n = \opn{depth}_A(\bar{I},M) - \inf(M).
\]
Hence, any maximal $M$-regular sequence contained in $\bar{I}$ has length $\opn{depth}_A(\bar{I},M) - \inf(M)$.
\end{proof}

\begin{prop}\label{prop:depthPIn}
Let $(A,\bar{\m})$ be a noetherian local DG-ring,
and let $M \in \cat{D}^{+}(A)$.
Given $\bar{\p} \in \opn{Spec}(\mrm{H}^0(A))$,
we have that
\[
\opn{depth}_A(\bar{\p},M) \le \opn{depth}_{A_{\bar{\p}}}(M_{\bar{\p}}).
\]
\end{prop}
\begin{proof}
Consider the DG-module $X = \mrm{R}\opn{Hom}_A(\mrm{H}^0(A)/\bar{\p},M)$.
By definition, we have that
\[
\opn{depth}_A(\bar{\p},M) = \inf(X).
\]
On the other hand, by \cite[Theorem 12.9.10]{YeBook} and adjunction, we have that
\begin{gather*}
X_{\bar{\p}} = \mrm{R}\opn{Hom}_A(\mrm{H}^0(A)/\bar{\p},M) \otimes^{\mrm{L}}_A A_{\bar{\p}} \cong\\
\mrm{R}\opn{Hom}_A(\mrm{H}^0(A)/\bar{\p},M \otimes^{\mrm{L}}_A A_{\bar{\p}}) \cong
\mrm{R}\opn{Hom}_{A_{\bar{\p}}}(\mrm{H}^0(A)_{\bar{\p}}/\bar{\p}\mrm{H}^0(A)_{\bar{\p}},M_{\bar{\p}}).
\end{gather*}
We see that 
\[
\opn{depth}_{A_{\bar{\p}}}(M_{\bar{\p}}) = \inf(X_{\bar{\p}}).
\]
The result now follows from the inequality $\inf(X) \le \inf(X_{\bar{\p}})$.
\end{proof}

Unfortunately, the corresponding inequality does not hold,
when depth is replaced by sequential depth, as the next example shows.
This example uses the notion of a trivial extension DG-ring,
defined in \cite[Section 1]{Jo}, 
and recalled, in our notation, in \cite[Section 7]{ShCM}.

\begin{exa}\label{exa:seq-depth}
Let $\k$ be a field,
let $B = \k[[x,y]]/(x\cdot y)$,
and let $M$ be the $B$-module $M = B/(x) = \k[[y]]$.
Let $A = B \skewtimes M[2]$ be a trivial extension DG-ring,
and consider the prime ideal $\bar{\q} = (y) \in \opn{Spec}(\mrm{H}^0(A))$.
Since the element $y \in \mrm{H}^0(A)$ is $M$-regular,
and $M = \mrm{H}^{\inf(A)}(A)$,
it follows that $\opn{seq.depth}_A(\bar{\q},A) \ge 1$.
However, the localization of $A$ at $\bar{\q}$ is quasi-isomorphic to a field,
so that 
\[
\opn{seq.depth}_{A_{\bar{\q}}}(A_{\bar{\q}}) = 0 < \opn{seq.depth}_A(\bar{\q},A).
\]
Notice that $\bar{\q} \notin \opn{Supp}(\mrm{H}^{\inf(A)}(A))$.
\end{exa}

The failure of the above is the reason why in most of the paper we would have to assume that our DG-rings satisfy
$\opn{Supp}(\mrm{H}^{\inf(A)}(A)) = \opn{Spec}(\mrm{H}^0(A))$.
We will return to discussing this example in Example \ref{exa:counter} below.
In what follows we shall use the following terminology.
\begin{dfn}
Let $A$ be a commutative noetherian DG-ring with bounded cohomology.
We say that $A$ has constant amplitude if $\opn{Supp}(\mrm{H}^{\inf(A)}(A)) = \opn{Spec}(\mrm{H}^0(A))$.
\end{dfn}

Let us explain the reason for this terminology.
The condition that $\opn{Supp}(\mrm{H}^{\inf(A)}(A)) = \opn{Spec}(\mrm{H}^0(A))$ 
is equivalent to saying that for any $\bar{\p} \in \opn{Spec}(\mrm{H}^0(A))$,
it holds that 
\[
\left(\mrm{H}^{\inf(A)}(A)\right)_{\bar{\p}} \cong \left(\mrm{H}^{\inf(A)}(A_{\bar{\p}})\right) \ncong 0,
\]
or equivalently that $\inf(A) = \inf(A_{\bar{\p}})$.
Since for any $\bar{\p} \in \opn{Spec}(\mrm{H}^0(A))$,
we have that $\sup(A_{\bar{\p}}) = 0$,
the condition that $\opn{Supp}(\mrm{H}^{\inf(A)}(A)) = \opn{Spec}(\mrm{H}^0(A))$ 
is thus equivalent to the fact that the function $\opn{Spec}(\mrm{H}^0(A)) \to \mathbb{N}$ given by $\bar{\p} \mapsto \amp(A_{\bar{\p}})$ is constant.

The next lemma was proved by Bass (\cite[Lemma 3.1]{Ba}) in the case where $M$ is a finitely generated module.
We will need this lemma for bounded below complexes with finitely generated cohomology.
The same proof of Bass essentially works in the more general case.
Because of the centrality of this lemma in what follows,
we give a full proof here.

\begin{lem}\label{lem:Bass}
Let $(A,\m)$ be a noetherian local ring,
let $M \in \cat{D}^{+}_{\mrm{f}}(A)$,
and let $\p \subseteq \m$ be a prime ideal,
such that $\dim(A/\p) = 1$.
Suppose that
\[
\opn{Ext}^n_A(A/\p,M) \ne 0.
\]
Then
\[
\opn{Ext}^{n+1}_A(A/\m,M) \ne 0.
\]
\end{lem}
\begin{proof}
Let us choose some $x \in \m$ such that $x \notin \p$,
and set $B = A/\p$. Let $C = B/xB$.
Observe that $B$ is an integral domain, and $0 \ne x \in B$.
This implies that $C$ is zero dimensional, 
so that $C$ has finite length.
The short exact sequence of $A$-modules
\[
0 \to B \xrightarrow{\cdot x} B \to C \to 0
\]
implies that there is an exact sequence
\[
\opn{Ext}^n_A(C,M) \to \opn{Ext}^n_A(B,M) \xrightarrow{\cdot x} \opn{Ext}^n_A(B,M) \to \opn{Ext}^{n+1}_A(C,M)
\]
By Nakayama's lemma, we know that the map $\opn{Ext}^n_A(B,M) \xrightarrow{\cdot x} \opn{Ext}^n_A(B,M)$ cannot be surjective,
so by assumption, we must have that $\opn{Ext}^{n+1}_A(C,M) \ne 0$.
Now, as $C$ is a module of finite length,
we may find some $A$-module $C' \subseteq C$,
such that $C/C' \cong A/\m$,
and such that $\ell(C')  < \ell(C)$.
This gives an exact sequence
\[
0 \to C' \to C \to C/C' \to 0
\]
which in turn implies that there is an exact sequence
\[
\opn{Ext}^{n+1}_A(C/C',M) \to \opn{Ext}^{n+1}_A(C,M) \to \opn{Ext}^{n+1}_A(C',M)
\]
Since $\opn{Ext}^{n+1}_A(C,M) \ne 0$,
at least one of $\opn{Ext}^{n+1}_A(C/C',M)$ and $\opn{Ext}^{n+1}_A(C',M)$ must be non-zero.
If $\opn{Ext}^{n+1}_A(C/C',M)$ we are done, and if not, 
we replace $C$ by $C'$ and repeat this process,
until we arrive to the conclusion that $\opn{Ext}^{n+1}_A(A/\m,M) \ne 0$.
\end{proof}

\begin{lem}\label{lem:generalBass}
Let $(A,\m)$ be a noetherian local ring,
let $M \in \cat{D}^{+}_{\mrm{f}}(A)$,
and let $\p \in \opn{Spec}(A)$ be such that $\dim(A/\p) = d$.
Suppose that
\[
\opn{Ext}^n_A(A/\p,M) \ne 0.
\]
Then
\[
\opn{Ext}^{n+d}_A(A/\m,M) \ne 0.
\]
\end{lem}
\begin{proof}
The proof is by induction on $d$.
There is nothing to prove if $d=0$, 
and the case where $d=1$ was shown in Lemma \ref{lem:Bass}.
Assume $d>1$.
Then we may find a prime ideal $\p \subsetneq \q \subsetneq \m$,
such that there are no primes between $\p$ and $\q$,
and such that $\dim(A/\q) = d-1$.
Since $\opn{Ext}^n_A(A/\p,M) \ne 0$,
it follows from \cite[Corollary 2.4]{Ba} that
\[
\opn{Ext}^{n}_{A_{\q}}(A_{\q}/\p A_{\q},M_{\q}) \ne 0.
\]
By Lemma \ref{lem:Bass},
this implies that
\[
\opn{Ext}^{n+1}_{A_{\q}}(A_{\q}/\q A_{\q},M_{\q}) \ne 0,
\]
which implies by \cite[Corollary 2.4]{Ba} that
\[
\opn{Ext}^{n+1}_{A}(A/\q,M) \ne 0.
\]
Since $\dim(A/\q) = d-1$,
by the induction hypothesis we deduce that
\[
\opn{Ext}^{n+d}_{A}(A/\m,M) \ne 0.
\]
\end{proof}

The next result generalizes \cite[Theorem 17.2]{Mat} to the DG setting.

\begin{thm}\label{thm:depth-bound}
Let $(A,\m)$ be a noetherian local DG-ring,
let $M \in \cat{D}^{+}_{\mrm{f}}(A)$,
let $\bar{\p} \in \opn{Ass}_A(M)$,
and suppose that $\p \in \opn{Supp}(\mrm{H}^{\inf(M)}(M))$.
Then
\[
\opn{seq.depth}_A(M) \le \dim(\mrm{H}^0(A)/\bar{\p}).
\]
\end{thm}
\begin{proof}
Since $\bar{\p} \in \opn{Ass}_A(M)$,
\[
\opn{depth}_{A_{\bar{\p}}}(M_{\bar{\p}}) = \inf(M_{\bar{\p}}) = \inf(M).
\]
By Proposition \ref{prop:depthPIn},
\[
\opn{depth}_A(\bar{\p},M) \le \opn{depth}_{A_{\bar{\p}}}(M_{\bar{\p}}) = \inf(M),
\]
so from Proposition \ref{prop:depth-formula} we get
\[
\opn{seq.depth}_A(\bar{\p},M) = \opn{depth}_A(\bar{\p},M) - \inf(M) \le 0.
\]
As this number is, by definition, non-negative,
we deduce that $\opn{seq.depth}_A(\bar{\p},M) = 0$.
Hence, $\opn{depth}_A(\bar{\p},M) = \inf(M)$,
so that
\[
\inf\left(\mrm{R}\opn{Hom}_A(\mrm{H}^0(A)/\bar{\p},M)\right) = \inf(M).
\]
Let us set $i = \inf(M)$.
The isomorphism
\[
\mrm{R}\opn{Hom}_A(\mrm{H}^0(A)/\bar{\p},M) \cong \mrm{R}\opn{Hom}_{\mrm{H}^0(A)}(\mrm{H}^0(A)/\bar{\p},\mrm{R}\opn{Hom}_A(\mrm{H}^0(A),M))
\]
implies that
\[
\opn{Ext}^{i}_{\mrm{H}^0(A)}(\mrm{H}^0(A)/\bar{\p},\mrm{R}\opn{Hom}_A(\mrm{H}^0(A),M)) \ne 0.
\]
It follows from \cite[Theorem 2.13]{Ye1} that
\[
\mrm{R}\opn{Hom}_A(\mrm{H}^0(A),M) \in \cat{D}^{+}_{\mrm{f}}(\mrm{H}^0(A)).
\]
Letting $d = \dim(\mrm{H}^0(A)/\bar{\p})$,
we deduce from Lemma \ref{lem:generalBass} that
\[
\opn{Ext}^{i+d}_{\mrm{H}^0(A)}(\mrm{H}^0(A)/\bar{\m},\mrm{R}\opn{Hom}_A(\mrm{H}^0(A),M)) \ne 0.
\]
Using adjunction again, 
this implies that
\[
\mrm{H}^{i+d}\left(\mrm{R}\opn{Hom}_A(\mrm{H}^0(A)/\bar{\m},M)\right) \ne 0,
\]
which shows that $\opn{depth}_A(M) \le \inf(M) + \dim(\mrm{H}^0(A)/\bar{\p})$.
Hence, Proposition \ref{prop:depth-formula} gives:
\[
\opn{seq.depth}_A(M) = \opn{depth}_A(M) - \inf(M) \le \dim(\mrm{H}^0(A)/\bar{\p}),
\]
as claimed.
\end{proof}

We will now use this result to study Cohen-Macaulay DG-rings.
In the next result we shall use the notions of associated primes of DG-modules and the set $W_0^A$ which were recalled in Section \ref{sec:regSeq}.

\begin{prop}\label{prop:height-of-ass}
Let $(A,\bar{\m})$ be a Cohen-Macaulay local DG-ring.
Then $\bar{\p} \in \opn{Ass}(A)$ if and only if 
$\bar{\p} \in W_0^A(A)$ if and only if
$\opn{ht}(\bar{\p}) = 0$.
\end{prop}
\begin{proof}
By definition, $\bar{\p} \in \opn{Ass}(A)$ if and only if $\depth_{A_{\bar{\p}}}(A_{\bar{\p}}) = \inf(A_{\bar{\p}})$.
On the other hand,
\[
\depth(A_{\bar{\p}}) = \opn{seq.depth}_{A_{\bar{\p}}}(A_{\bar{\p}}) + \inf(A_{\bar{\p}}),
\]
so we deduce that $\bar{\p} \in \opn{Ass}(A)$ if and only if $\opn{seq.depth}_{A_{\bar{\p}}}(A_{\bar{\p}}) = 0$.
But $A_{\bar{\p}}$ is a Cohen-Macaulay DG-ring,
so that
\[
\opn{seq.depth}_{A_{\bar{\p}}}(A_{\bar{\p}}) = \dim(\mrm{H}^0(A_{\bar{\p}})) = \dim(\mrm{H}^0(A)_{\bar{\p}}).
\]
We see that $\bar{\p} \in \opn{Ass}(A)$ if and only if $\dim(\mrm{H}^0(A)_{\bar{\p}}) = 0$, as claimed.
Similarly, according to \cite[Equation (5.12)]{ShCM},
$\bar{\p} \in W_0^A(A)$ if and only if $\lcdim_{A_{\bar{\p}}}(A_{\bar{\p}}) = \sup(A_{\bar{\p}})$,
but by \cite[Equation (2.6)]{ShCM}, we have that
\[
\lcdim_{A_{\bar{\p}}}(A_{\bar{\p}}) = \dim(\mrm{H}^0(A)_{\bar{\p}}),
\]
and $\sup(A_{\bar{\p}}) = 0$, which implies the result.
\end{proof}

It is well known that Cohen-Macaulay local rings are equidimensional.
This is false for Cohen-Macaulay local DG-rings,
because, for example, any noetherian local ring $A$ which has a dualizing complex can be realized as $\mrm{H}^0(B)$,
where $B$ is a Cohen-Macaulay DG-ring. 
However, we now show that if we assume that the bottom cohomology of $B$ has full support, 
then Cohen-Macaulay local DG-rings are equidimensional.

\begin{cor}\label{cor:equid}
Let $(A,\bar{\m})$ be a Cohen-Macaulay local DG-ring with constant amplitude.
Then $\mrm{H}^0(A)$ is equidimensional.
\end{cor}
\begin{proof}
Since $\mrm{H}^0(A)$ is local,
it is enough to show that any prime ideal $\bar{\p} \in \opn{Spec}(\mrm{H}^0(A))$
with $\opn{ht}(\bar{\p}) = 0$ satisfies $\dim(\mrm{H}^0(A)/\bar{\p}) = \dim(\mrm{H}^0(A))$.
By Proposition \ref{prop:height-of-ass},
any such $\bar{\p}$ is an associated prime of $A$,
and by assumption, $\bar{\p} \in \opn{Supp}(\mrm{H}^{\inf(A)}(A))$.
Hence, by Theorem \ref{thm:depth-bound} applied to $A$,
\[
\opn{seq.depth}_A(A) \le \dim(\mrm{H}^0(A)/\bar{\p})
\]
Since $A$ is Cohen-Macaulay, 
\[
\opn{seq.depth}_A(A) = \dim(\mrm{H}^0(A)),
\]
which shows that
\[
\dim(\mrm{H}^0(A)/\bar{\p}) = \dim(\mrm{H}^0(A)).
\]
\end{proof}

\begin{lem}\label{lem:seq-zero-in-ass}
Let $(A,\bar{\m})$ be a local noetherian DG-ring,
let $M \in \cat{D}^{\mrm{b}}_{\mrm{f}}(A)$,
and let $\bar{I} \subseteq \mrm{H}^0(A)$ be an ideal.
If $\opn{seq.depth}(\bar{I},M) = 0$,
then there exists a prime ideal $\bar{\p} \in \opn{Ass}_A(M)$
such that $\bar{I} \subseteq \bar{\p}$.
\end{lem}
\begin{proof}
According to \cite[Proposition 5.9]{ShCM},
the set $\opn{Ass}_A(M)$ is finite.
Hence, by the prime avoidance lemma,
if for all $\bar{\p} \in \opn{Ass}_A(M)$ we have that $\bar{I} \nsubseteq \bar{\p}$,
then there exists $\bar{x} \in \bar{I}$ such that for all $\bar{\p} \in \opn{Ass}_A(M)$,
$\bar{x} \notin \bar{\p}$.
Then, by \cite[Proposition 5.13]{ShCM},
the element $\bar{x}$ is $M$-regular,
contradicting the fact that $\opn{seq.depth}(\bar{I},M) = 0$.
\end{proof}

\begin{lem}\label{lem:strongRegular}
Let $(A,\bar{\m})$ be a Cohen-Macaulay local DG-ring with constant amplitude.
Let $\bar{I} \subseteq \mrm{H}^0(A)$ be an ideal such that $\opn{seq.depth}(\bar{I},A) > 0$.
Then there exists an element $\bar{x} \in \bar{I}$
such that $\bar{x}$ is $A$-regular,
and moreover, $\dim(\mrm{H}^0(K(A;\bar{x}))) = \dim(\mrm{H}^0(A))-1$.
\end{lem}
\begin{proof}
First, we claim that $\opn{ht}(\bar{I}) > 0$.
If that is not the case,
then there is some $\bar{\p} \in \opn{Spec}(\mrm{H}^0(A))$,
such that $\opn{ht}(\bar{\p}) = 0$ and $\bar{I} \subseteq \bar{\p}$.
By Proposition \ref{prop:height-of-ass},
$\bar{\p} \in \opn{Ass}(A)$.
Hence, using Proposition \ref{prop:depthPIn} we get:
\[
\opn{depth}_A(\bar{\p},A)\le \opn{depth}_{A_{\bar{\p}}}(A_{\bar{\p}}) = \inf(A_{\bar{\p}}) = \inf(A),
\]
so from Proposition \ref{prop:depth-formula} we get that
$\opn{seq.depth}_A(\bar{\p},A) = 0$.
This means that for any $\bar{x} \in \bar{\p}$, 
we have that $\bar{x}$ is not $A$-regular,
contradicting the assumption that $\opn{seq.depth}(\bar{I},A) > 0$.
Thus, $\opn{ht}(\bar{I}) > 0$,
which implies that $\bar{I}$ is not contained in any minimal prime ideal of $\mrm{H}^0(A)$.
By Proposition \ref{prop:height-of-ass},
we see that $\bar{I}$ is not contained in $\bar{\p}$ for any $\bar{\p} \in \opn{Ass}_A(A)$,
and for any $\bar{\p} \in W_0^A(A)$.
By the prime avoidance lemma, 
we can find $\bar{x} \in \bar{I}$,
such that $\bar{x} \notin \bar{\p}$ for any minimal prime $\bar{\p}$ of $\mrm{H}^0(A)$.
It follows from \cite[Proposition 5.13]{ShCM} that $\bar{x}$ is $A$-regular,
and from \cite[Proposition 5.16]{ShCM} that $\dim(\mrm{H}^0(K(A;\bar{x}))) = \dim(\mrm{H}^0(A))-1$.
\end{proof}

\begin{lem}\label{lem:supportKoszul}
Let $(A,\bar{\m})$ be a local noetherian DG-ring with bounded cohomology and constant amplitude.
Let $\bar{x} \in \bar{\m}$ be an $A$-regular element,
and let $B = K(A;\bar{x})$.
Then $B$ also has constant amplitude.
\end{lem}
\begin{proof}
Let $n = \inf(A)$. 
By \cite[Lemma 2.13]{Mi}, 
we have that $\inf(B) = n$.
We know that $\mrm{H}^0(B) = \mrm{H}^0(A)/\bar{x}$.
Given $\bar{\q} \in \opn{Spec}(\mrm{H}^0(B))$,
let $\bar{\p}$ be its preimage in $\opn{Spec}(\mrm{H}^0(A))$ along the surjection $\mrm{H}^0(A) \to \mrm{H}^0(A)/\bar{x}$.
Then we have that $\bar{x} \in \bar{\p}$.
Since the sequence
\[
0 \to \mrm{H}^n(A) \xrightarrow{\cdot \bar{x}} \mrm{H}^n(A) \to \mrm{H}^n(B)
\]
is exact, localizing it at $\bar{\p}$ we obtain an exact sequence
\[
0 \to \mrm{H}^n(A)_{\bar{\p}} \xrightarrow{\cdot \frac{\bar{x}}{1}} \mrm{H}^n(A)_{\bar{\p}} \to \mrm{H}^n(B)_{\bar{\p}}.
\]
The fact that $\bar{x} \in \bar{\p}$ implies that $0\ne \frac{\bar{x}}{1} \in \bar{\p}A_{\bar{\p}}$,
so Nakayama's lemma implies that  $\mrm{H}^n(B)_{\bar{\p}} \ne 0$.
The result then follows from the fact that $B_{\bar{\p}} \cong B_{\bar{\q}}$. 
\end{proof}

\begin{prop}\label{prop:assNoReg}
Let $(A,\bar{\m})$ be a local noetherian DG-ring with bounded cohomology and constant amplitude.
Let $\bar{x} \in \bar{\m}$ be an $A$-regular element.
Then $\bar{x}$ is not contained in any associated prime of $A$.
\end{prop}
\begin{proof}
Let $\bar{\p} \in \opn{Spec}(\mrm{H}^0(A))$,
and suppose that $\bar{x} \in \bar{\p}$.
Then $\opn{seq.depth}(\bar{\p},A) > 0$.
Hence,
\[
\depth_{A_{\bar{\p}}}(A_{\bar{\p}}) \ge \opn{depth}(\bar{\p},A) = \opn{seq.depth}(\bar{\p},A) + \inf(A) > \inf(A) = \inf(A_{\bar{\p}}).
\]
Which implies that $\bar{\p} \notin \opn{Ass}(A)$.
\end{proof}

The next result is a very special case of the main result of this paper.
We need to prove this special case here, as it is required in the sequel.

\begin{prop}\label{prop:koszulCM}
Let $(A,\bar{\m})$ be a Cohen-Macaulay local DG-ring with constant amplitude.
Let $\bar{x} \in \bar{\m}$ be an $A$-regular element, such that
\[
\dim(\mrm{H}^0(K(A;\bar{x}))) = \dim(\mrm{H}^0(A))-1.
\]
Then the local DG-ring $K = K(A;\bar{x})$ is also Cohen-Macaulay.
\end{prop}
\begin{proof}
Given $\bar{\p} \in \opn{Spec}(\mrm{H}^0(K)) = \opn{Spec}(\mrm{H}^0(A)/\bar{x})$,
we must show that $K_{\bar{\p}}$ is local-Cohen-Macaulay.
If $K_{\bar{\p}} \cong 0$, there is nothing to prove.
Otherwise, let $\bar{\q}$ be the inverse image of $\bar{\p}$ under the map $\mrm{H}^0(A) \to \mrm{H}^0(A)/(\bar{x})$.
Then $\bar{x} \in \bar{\q}$.
By Proposition \ref{prop:kosLocal},
we know that
\[
K_{\bar{\p}} \cong K(A_{\bar{\q}};\bar{x}/1).
\]
If $\opn{ht}(\bar{\q}) = 0$,
so that $\dim(\mrm{H}^0(A_{\bar{\q}})) =0$,
then $\dim(\mrm{H}^0(K_{\bar{\p}})) =0$,
so by \cite[Proposition 4.8]{ShCM}, 
$K_{\bar{\p}}$ is local-Cohen-Macaulay.
Suppose $\opn{ht}(\bar{\q}) > 0$.
Since $\bar{x}$ is $\mrm{H}^{\inf(A)}(A)$-regular,
it follows that $\bar{x}/1$ is $\mrm{H}^{\inf(A_{\bar{\q})}}(A_{\bar{\q}})$-regular,
so that $\bar{x}/1$ is $A_{\bar{\q}}$-regular.
Hence,
\[
\opn{seq.depth}_{K_{\bar{\p}}}(K_{\bar{\p}}) = \opn{seq.depth}_{A_{\bar{\q}}}(A_{\bar{\q}}) - 1 = \dim(\mrm{H}^0(A_{\bar{\q}})) - 1.
\]
It is thus enough to show that
\[
\dim(\mrm{H}^0(K_{\bar{\p}})) =  \dim(\mrm{H}^0(A_{\bar{\q}})/(\bar{x})) = \dim(\mrm{H}^0(A_{\bar{\q}}))-1.
\]
If that is not the case,
so that $\dim(\mrm{H}^0(A_{\bar{\q}})/(\bar{x})) = \dim(\mrm{H}^0(A_{\bar{\q}}))$,
then $\bar{x}$ must be contained in some minimal prime ideal $\bar{\n}$ contained in $\bar{\q}$.
However, since $A_{\bar{\q}}$ is Cohen-Macaulay,
by Proposition \ref{prop:height-of-ass},
any such $\bar{\n}$ is an associated prime of $A_{\bar{\q}}$,
and by Proposition \ref{prop:assNoReg},
$\bar{x}$ is not contained in any such $\bar{\n}$,
which shows that $K_{\bar{\p}}$ is local-Cohen-Macaulay.
\end{proof}

Here is the main result of this section.

\begin{thm}\label{thm:seq-depth}
Let $(A,\bar{\m})$ be a Cohen-Macaulay local DG-ring with constant amplitude.
Then for any ideal $\bar{I} \subseteq \mrm{H}^0(A)$,
there is an equality
\[
\opn{seq.depth}_A(\bar{I},A) = \dim(\mrm{H}^0(A)) - \dim(\mrm{H}^0(A)/\bar{I}).
\]
\end{thm}
\begin{proof}
We prove this by induction on $\opn{seq.depth}_A(\bar{I},A)$.
If $\opn{seq.depth}(\bar{I},A) = 0$, 
then by Lemma \ref{lem:seq-zero-in-ass} there is some $\bar{\p} \in \opn{Ass}(A)$ such that $\bar{I} \subseteq \bar{\p}$.
In that case, by Proposition \ref{prop:height-of-ass} and Corollary \ref{cor:equid} we have
\[
\dim(\mrm{H}^0(A)/\bar{I}) \ge \dim(\mrm{H}^0(A)/\bar{\p}) = \dim(\mrm{H}^0(A)),
\]
so that 
\[
\dim(\mrm{H}^0(A)/\bar{I}) = \dim(\mrm{H}^0(A)),
\]
as needed in this case.
Assume now that $\opn{seq.depth}(\bar{I},A) > 0$.
By Lemma \ref{lem:strongRegular},
there is some $\bar{x} \in \bar{I}$ such that $\bar{x}$ is $A$-regular,
and $\dim(\mrm{H}^0(K(A;\bar{x}))) = \dim(\mrm{H}^0(A))-1$.
Let $B = K(A;\bar{x})$.
By Proposition \ref{prop:koszulCM},
$B$ is Cohen-Macaulay,
and by Lemma \ref{lem:supportKoszul}, $B$ has constant amplitude.
Let $\bar{J}$ be the image of $\bar{I}$ in $\mrm{H}^0(B) = \mrm{H}^0(A)/(\bar{x})$.
Then 
\[
\opn{seq.depth}_B(\bar{J},B) = \opn{seq.depth}_A(\bar{I},A) -1.
\]
By the induction hypothesis,
\[
\opn{seq.depth}_B(\bar{J},B) = \dim(\mrm{H}^0(B))-\dim(\mrm{H}^0(B)/\bar{J}).
\]
Hence,
\begin{gather*}
\opn{seq.depth}(\bar{I},A)  = \dim(\mrm{H}^0(B))-\dim(\mrm{H}^0(B)/\bar{J}) + 1 =\\
\dim(\mrm{H}^0(A))-1 - \dim(\mrm{H}^0(A)/\bar{I}) + 1 =\dim(\mrm{H}^0(A))-\dim(\mrm{H}^0(A)/\bar{I}).
\end{gather*}
\end{proof}

We finish this section with the next result which connects depth and the Koszul complex over commutative DG-rings:

\begin{prop}\label{prop:depthKoszul}
Let $A$ be a commutative noetherian DG-ring,
let $\bar{I} \subseteq \mrm{H}^0(A)$ be a proper ideal,
and let $\bar{a}_1,\dots,\bar{a}_n$ be a sequence of elements of $\mrm{H}^0(A)$ that generates $\bar{I}$.
Then for any $M \in \cat{D}^{+}(A)$,
there is an equality
\[
\opn{depth}_A(\bar{I},M) = \inf\left(M\otimes^{\mrm{L}}_A K(A;\bar{a}_1,\dots,\bar{a}_n)\right) + n.
\]
\end{prop}
\begin{proof}
By \cite[Proposition 3.3]{ShInj} we have an equality
\[
\inf\left(M\otimes^{\mrm{L}}_A K(A;\bar{a}_1,\dots,\bar{a}_n)\right) =
\inf\left(\mrm{R}\opn{Hom}_A(\mrm{H}^0(A),M\otimes^{\mrm{L}}_A K(A;\bar{a}_1,\dots,\bar{a}_n))\right). 
\]
By \cite[Theorem 12.10.14]{YeBook} and adjunction,
there are isomorphisms
\begin{gather*}
\mrm{R}\opn{Hom}_A(\mrm{H}^0(A),M\otimes^{\mrm{L}}_A K(A;\bar{a}_1,\dots,\bar{a}_n)) \cong\\
\mrm{R}\opn{Hom}_A(\mrm{H}^0(A),M)\otimes^{\mrm{L}}_A K(A;\bar{a}_1,\dots,\bar{a}_n) \cong \\
\mrm{R}\opn{Hom}_A(\mrm{H}^0(A),M)\otimes^{\mrm{L}}_{\mrm{H}^0(A)} \mrm{H}^0(A) \otimes^{\mrm{L}}_A K(A;\bar{a}_1,\dots,\bar{a}_n).
\end{gather*}
Applying Proposition \ref{prop:KoszulChange} to the map $A \to \mrm{H}^0(A)$,
we know that
\[
\mrm{H}^0(A) \otimes^{\mrm{L}}_A K(A;\bar{a}_1,\dots,\bar{a}_n) \cong K(\mrm{H}^0(A);\bar{a}_1,\dots,\bar{a}_n).
\]
We deduce that
\begin{gather*}
\inf\left(M\otimes^{\mrm{L}}_A K(A;\bar{a}_1,\dots,\bar{a}_n)\right) =\\
\inf\left(\mrm{R}\opn{Hom}_A(\mrm{H}^0(A),M)\otimes^{\mrm{L}}_{\mrm{H}^0(A)} K(\mrm{H}^0(A);\bar{a}_1,\dots,\bar{a}_n) \right).
\end{gather*}
To compute the latter,
we may invoke \cite[Theorem I]{FI} over the noetherian ring $\mrm{H}^0(A)$,
and deduce that
\[
\inf\left(M\otimes^{\mrm{L}}_A K(A;\bar{a}_1,\dots,\bar{a}_n)\right) = \opn{depth}_{\mrm{H}^0(A)}\left(\bar{I},\mrm{R}\opn{Hom}_A(\mrm{H}^0(A),M)\right) -n
\]
Note that the formula here is slightly different than the one in \cite{FI},
because we are using cohomological grading.
By definition, 
we have that
\begin{gather*}
\opn{depth}_{\mrm{H}^0(A)}\left(\bar{I},\mrm{R}\opn{Hom}_A(\mrm{H}^0(A),M)\right) = \\
\inf\left(\mrm{R}\opn{Hom}_{\mrm{H}^0(A)}(\mrm{H}^0(A)/\bar{I},\mrm{R}\opn{Hom}_A(\mrm{H}^0(A),M))\right),
\end{gather*}
and by adjunction we have that
\[
\mrm{R}\opn{Hom}_{\mrm{H}^0(A)}\left(\mrm{H}^0(A)/\bar{I},\mrm{R}\opn{Hom}_A(\mrm{H}^0(A),M)\right) \cong
\mrm{R}\opn{Hom}_A(\mrm{H}^0(A)/\bar{I},M).
\]
Since by definition $\inf\left(\mrm{R}\opn{Hom}_A(\mrm{H}^0(A)/\bar{I},M)\right) = \opn{depth}_A(\bar{I},M)$,
we deduce the result.
\end{proof}

\section{Koszul complexes over Cohen-Macaulay DG-rings}\label{sec:KosCMDG}

In this section we prove the main result of this paper,
Theorem \ref{thm:main}.
Before that, we need the following lemma.

\begin{lem}\label{lem:supportCompletion}
Let $(A,\m)$ be a noetherian local ring,
and let $M$ be a finitely generated $A$-module.
Suppose that
\[
\opn{Supp}_A(M) = \opn{Spec}(A),
\]
and let $\widehat{A}$ and $\widehat{M}$ denote the $\m$-adic completions of $A$ and $M$.
Then
\[
\opn{Supp}_{\widehat{A}}(\widehat{M}) = \opn{Spec}(\widehat{A}).
\]
\end{lem}
\begin{proof}
Denote by $\tau:A \to \widehat{A}$ the completion map.
Given $\q \in \opn{Spec}(\widehat{A})$,
let $\p = \tau^{-1}(\q)$.
Since $M$ is finitely generated, 
we have that
\[
\widehat{M}_{\q} \cong (M\otimes_A \widehat{A}) \otimes_{\widehat{A}} \widehat{A}_{\q} \cong M\otimes_A \widehat{A}_{\q}.
\]
Since the map $A \to \widehat{A}_{\q}$ factors as $A \to A_{\p} \to \widehat{A}_{\q}$,
we have that
\[
M\otimes_A \widehat{A}_{\q} \cong M\otimes_A A_{\p}\otimes_{A_{\p}} \widehat{A}_{\q} \cong M_{\p}\otimes_{A_{\p}} \widehat{A}_{\q}.
\]
Since $M_{\p} \ne 0$, 
and as the map $A_{\p} \to \widehat{A}_{\q}$ is faithfully flat,
we deduce that $(\widehat{M})_{\q} \ne 0$.
\end{proof}

Here is the main result of this paper.

\begin{thm}\label{thm:main}
Let $A$ be a Cohen-Macaulay DG-ring with constant amplitude, 
and let $\bar{a}_1,\dots,\bar{a}_n \in \mrm{H}^0(A)$ be any finite sequence of elements.
Then the Koszul complex $K = K(A;\bar{a}_1,\dots,\bar{a}_n)$ is a Cohen-Macaulay DG-ring.
\end{thm}
\begin{proof}
Given $\bar{\p} \in \opn{Spec}(\mrm{H}^0(K))$,
we must show that $K_{\bar{\p}}$ is local-Cohen-Macaulay.
Let $\bar{\q}$ be the inverse image of $\bar{\p}$ in $\mrm{H}^0(A)$.
By Proposition \ref{prop:kosLocal},
we know that
\[
K_{\bar{\p}} \cong K(A_{\bar{\q}};\bar{a}_1/1,\dots,\bar{a}_n/1).
\]
Replacing $A$ by $A_{\bar{\q}}$,
we may thus assume without loss of generality that $A$ is local.
Let us denote its maximal ideal by $\bar{\m}$.
After replacing $A$ by $A_{\bar{\q}}$,
it is enough to show that $K = K(A;\bar{a}_1,\dots,\bar{a}_n)$ is local-Cohen-Macaulay.
Moreover, if for some $1\le i \le n$, $\bar{a}_i \notin \bar{\m}$,
then $K_{\bar{\p}} \cong 0$,
so we might as well assume that $\bar{a}_i \in \bar{\m}$ for all $1 \le i \le n$.

Letting $\bar{\n}$ denote the image of $\bar{\m}$ in $\mrm{H}^0(K)$,
by \cite[Proposition 4.6]{ShCM},
$K$ is local-Cohen-Macaulay if and only if
the derived completion $\mrm{L}\Lambda(K,\bar{\n})$ is local-Cohen-Macaulay.
Letting $\widehat{\bar{a}}_1,\dots,\widehat{\bar{a}}_n$ denote the images of
$\bar{a}_1,\dots,\bar{a}_n$ in the $\bar{\m}$-adic completion of $\mrm{H}^0(A)$,
according to Theorem \ref{thm:koszulCompletion},
\[
\mrm{L}\Lambda( K,\bar{\n}) \cong K( \mrm{L}\Lambda(A,\bar{\m}); \widehat{\bar{a}}_1,\dots,\widehat{\bar{a}}_n). 
\]
We further note that by \cite[Proposition 1.7]{ShCM},
we have that $\inf(\mrm{L}\Lambda(A,\bar{\m})) = \inf(A)$,
and $\mrm{H}^{\inf(A)}(\mrm{L}\Lambda(A,\bar{\m}))$ is equal to the $\bar{\m}$-adic completion of $\mrm{H}^{\inf(A)}(A)$.
Hence, by Lemma \ref{lem:supportCompletion},
we see that
\[
\opn{Supp}(\mrm{H}^{\inf(\mrm{L}\Lambda(A,\bar{\m}))}(\mrm{L}\Lambda(A,\bar{\m}))) = \opn{Spec}(\mrm{H}^0(\mrm{L}\Lambda(A,\bar{\m}))).
\]
Hence, we may replace $A$ by $\mrm{L}\Lambda(A,\bar{\m})$,
so we may assume without loss of generality that $A$ is both local and derived $\bar{\m}$-adically complete.

By \cite[Proposition 7.21]{ShInj},
this implies that $A$ has a dualizing DG-module.
Let $R$ be a dualizing DG-module over $A$.
Define $D := \mrm{R}\opn{Hom}_A(K,R)$.
Since $\mrm{H}^0(A) \to \mrm{H}^0(K)$ is surjective,
by \cite[Proposition 7.5]{Ye1},
the DG-module $D$ is a dualizing DG-module over $K$.
Hence, it is enough to show that $\amp(K) = \amp(D)$.

To do this, we will explicitly compute these two numbers.
Since $K$ is a non-positive DG-ring,
we have that $\amp(K) = -\inf(K)$.
By Proposition \ref{prop:depthKoszul},
we know that
\[
-\inf(K) = n-\opn{depth}_A(\bar{I},A),
\]
and by Proposition \ref{prop:depth-formula} and Theorem \ref{thm:seq-depth},
we have that
\begin{gather*}
n-\opn{depth}_A(\bar{I},A) =
n-\left(\opn{seq.depth}_A(\bar{I},A)+\inf(A)\right) =\\
n-\left(\dim(\mrm{H}^0(A)-\dim(\mrm{H}^0(A)/\bar{I})+\inf(A) \right).
\end{gather*}
It follows that
\begin{equation}\label{eqn:ampKoszul}
\amp(K) = n-\dim(\mrm{H}^0(A))+\dim(\mrm{H}^0(A)/\bar{I})-\inf(A).
\end{equation}

To compute $\amp(D)$,
we may forget its $K$-structure,
and treat it as a DG-module over $A$.
Since by Proposition \ref{prop:compact},
$K$ is compact over $A$,
by \cite[Theorem 14.1.22]{YeBook},
we have that
\[
D = \mrm{R}\opn{Hom}_A(K,R) \cong \mrm{R}\opn{Hom}_A(K,A)\otimes^{\mrm{L}}_A R.
\]
By Proposition \ref{prop:dualKosz}, 
we see that
\[
\mrm{R}\opn{Hom}_A(K,A)\otimes^{\mrm{L}}_A R \cong K[-n]\otimes^{\mrm{L}}_A R.
\]
We deduce that $\amp(D) = \amp(K\otimes^{\mrm{L}}_A R)$.
To compute the latter,
let us normalize $R$,
so that $\inf(R) = -\dim(\mrm{H}^0(A))$.
Since $A$ is Cohen-Macaulay,
we know that $\amp(R) = \amp(A)$,
so that $\sup(R) = \amp(A)-\dim(\mrm{H}^0(A))$.
It follows from Nakayama's lemma that
\begin{equation}\label{eqn:supKR}
\sup(K\otimes^{\mrm{L}}_A R) = \sup(R) = \amp(A)-\dim(\mrm{H}^0(A)),
\end{equation}
while by Proposition \ref{prop:depthKoszul},
we have that
\begin{equation}\label{eqn:infKR}
\inf(K\otimes^{\mrm{L}}_A R) = \opn{depth}_A(\bar{I},R)-n.
\end{equation}

By \cite[Proposition 7.5]{Ye1},
the complex $\mrm{R}\opn{Hom}_A(\mrm{H}^0(A),R)$ is a dualizing complex over $\mrm{H}^0(A)$,
and by \cite[Proposition 3.3]{ShInj},
we know that 
\[
\inf\left(\mrm{R}\opn{Hom}_A(\mrm{H}^0(A),R)\right) = \inf(R) = -\dim(\mrm{H}^0(A)).
\]
It follows that $\mrm{R}\opn{Hom}_A(\mrm{H}^0(A),R)$ is a normalized dualizing complex,
in the sense of \cite[tag 0A7M]{SP},
so by \cite[tag 0A7N]{SP},
the complex
\[
\mrm{R}\opn{Hom}_{\mrm{H}^0(A)}(\mrm{H}^0(A)/\bar{I},\mrm{R}\opn{Hom}_A(\mrm{H}^0(A),R)) \cong \mrm{R}\opn{Hom}_A(\mrm{H}^0(A)/\bar{I},R)
\]
is a normalized dualizing complex over $\mrm{H}^0(A)/\bar{I}$.
This implies that
\[
\inf\left(\mrm{R}\opn{Hom}_A(\mrm{H}^0(A)/\bar{I},R)\right) = -\dim(\mrm{H}^0(A)/\bar{I}),
\]
which, by definition, shows that
\[
\opn{depth}_A(\bar{I},R) = -\dim(\mrm{H}^0(A)/\bar{I}).
\]
Combining this with (\ref{eqn:supKR}) and (\ref{eqn:infKR}),
we obtain:
\begin{gather*}
\amp(D) = \amp(K\otimes^{\mrm{L}}_A R) = \sup(K\otimes^{\mrm{L}}_A R) - \inf(K\otimes^{\mrm{L}}_A R) =\\
\amp(A)-\dim(\mrm{H}^0(A)) - \left(\opn{depth}_A(\bar{I},R)-n \right) =\\
\amp(A)-\dim(\mrm{H}^0(A)) - \left(-\dim(\mrm{H}^0(A)/\bar{I}) - n\right) = \\
n - \dim(\mrm{H}^0(A)) + \dim(\mrm{H}^0(A)/\bar{I}) - \inf(A),
\end{gather*}
which is exactly (\ref{eqn:ampKoszul}),
proving that $K$ is local-Cohen-Macaulay.
\end{proof}

As an important particular case of Theorem \ref{thm:main} we obtain:
\begin{cor}\label{cor:kosRing}
Let $A$ be a Cohen-Macaulay ring,
and let $a_1,\dots,a_n \in A$ be any finite sequence of elements.
Then the Koszul complex $K = K(A;a_1,\dots,a_n)$ is a Cohen-Macaulay DG-ring.
\end{cor}

\begin{rem}\label{rem:condition}
It is natural to ask if the assumption that $A$ has constant amplitude is necessary in Theorem \ref{thm:main}.
As Example \ref{exa:counter} below shows,
the theorem is false without this assumption.
We further remark that the proof of the theorem shows that one may assume slightly less,
namely, it is enough to assume that for any maximal ideal $\bar{\m} \in \opn{Spec}(\mrm{H}^0(A))$,
there is an equality
\begin{equation}\label{eqn:otherCond}
\opn{Supp}(\mrm{H}^{\inf(A_{\bar{\m}})}(A_{\bar{\m}})) = \opn{Spec}(\mrm{H}^0(A_{\bar{\m}})).
\end{equation}
In other words, it is enough to assume that the localizations of $A$ at maximal ideals have constant amplitude.
Since $A_{\bar{\m}}$ is Cohen-Macaulay,
by \cite[Proposition 4.11]{ShCM},
it holds that
\[
\dim(\mrm{H}^{\inf(A_{\bar{\m}})}(A_{\bar{\m}})) = \dim(\mrm{H}^0(A_{\bar{\m}})).
\]
Hence, we deduce (for instance, by \cite[Proposition 8.5]{ShCM})
that the condition (\ref{eqn:otherCond}) always holds if for each 
$\bar{\m} \in \opn{Spec}(\mrm{H}^0(A))$,
the local ring $\mrm{H}^0(A_{\bar{\m}})$ has an irreducible spectrum;
equivalently, if every maximal ideal in $\mrm{H}^0(A)$ contains a unique minimal prime ideal.
In particular, this is the case if $\mrm{H}^0(A)$ contains a unique minimal prime ideal.
\end{rem}

\begin{exa}\label{exa:counter}
As in Example \ref{exa:seq-depth},
Let $\k$ be a field,
let $B = \k[[x,y]]/(x\cdot y)$,
and let $M$ be the $B$-module $M = B/(x) = \k[[y]]$.
Consider again the trivial extension DG-ring $A = B \skewtimes M[2]$.
We saw that $\opn{seq.depth}_A(A) \ge 1$,
and as $\mrm{H}^0(A) = B$,
we see that $\dim(\mrm{H}^0(A)) = 1$,
so that $A$ is local-Cohen-Macaulay.
The two non-maximal prime ideals of $A$ are $\bar{\p} = (x)$ and $\bar{\q} = (y)$,
and both of them are of height $0$,
which implies by \cite[Proposition 4.8]{ShCM},
that the localizations $A_{\bar{\p}}$ and $A_{\bar{\q}}$ are local-Cohen-Macaulay,
so that $A$ is Cohen-Macaulay.

Consider the Koszul complex $K = K(A;y)$.
Then it holds that
\[
K \cong \k[[x]] \skewtimes \k[2].
\]
Hence, every element of $(x)$, 
the maximal ideal of $\mrm{H}^0(K)$ is not $K$-regular,
so we have that $\opn{seq.depth}_K(K) = 0$,
but $\dim(\mrm{H}^0(K)) = 1$.
It follows that $K$ is not Cohen-Macaulay.
\end{exa}

\begin{cor}
Let $B$ be a commutative noetherian ring which is a quotient of a Cohen-Macaulay ring.
Then there exists a Cohen-Macaulay DG-ring $A$ such that $\mrm{H}^0(A) \cong B$.
\end{cor}
\begin{proof}
Let $C$ be a Cohen-Macaulay ring,
such that there is an ideal $I \subseteq C$ with $C/I \cong B$.
Assume $I = (a_1,\dots,a_n)$,
Then by Theorem \ref{thm:main}, 
$A=K(C;a_1,\dots,a_n)$ is a Cohen-Macaulay DG-ring,
and $\mrm{H}^0(A)= C/(a_1,\dots,a_n) \cong B$.
\end{proof}

We finish this section with a corresponding result for Gorenstein DG-rings:
\begin{thm}\label{thm:Gorenstein}
Let $A$ be a commutative noetherian DG-ring with bounded cohomology,
and let $\bar{a}_1,\dots,\bar{a}_n \in \mrm{H}^0(A)$ be,
such that $(\bar{a}_1,\dots,\bar{a}_n) \subseteq \mrm{H}^0(A)$ is a proper ideal.
\leavevmode
\begin{enumerate}[wide, labelindent=0pt]
\item If $A$ is a Gorenstein DG-ring, then $K = K(A;\bar{a}_1,\dots,\bar{a}_n)$ is a Gorenstein DG-ring.
\item Conversely, if $K(A;\bar{a}_1,\dots,\bar{a}_n)$ is a Gorenstein DG-ring,
then for any 
\[
\bar{\p} \in V(\bar{a}_1,\dots,\bar{a}_n) = \{\bar{\p} \in \opn{Spec}(\mrm{H}^0(A)) \mid (\bar{a}_1,\dots,\bar{a}_n) \subseteq \bar{\p}\},
\]
the localization $A_{\bar{\p}}$ is a Gorenstein DG-ring.
\item In particular, if $(\bar{a}_1,\dots,\bar{a}_n) \subseteq \opn{rad}(\mrm{H}^0(A))$,
then $A$ is Gorenstein if and only if $K$ is Gorenstein.
\end{enumerate}
\end{thm}
\begin{proof}
\leavevmode
\begin{enumerate}[wide, labelindent=0pt]
\item Similarly to the proof of Theorem \ref{thm:main},
using Proposition \ref{prop:kosLocal},
we may reduce to the case where $(A,\bar{\m})$ is a Gorenstein local DG-ring,
and $\bar{a}_1\dots,\bar{a}_n \in \bar{\m}$. 
Then $A$ is a dualizing DG-module over itself, 
so by \cite[Proposition 7.5]{Ye1},
the DG-module $\mrm{R}\opn{Hom}_A(K,A)$ is a dualizing DG-module over $K$,
and since by Proposition \ref{prop:dualKosz} it is isomorphic to a shift of $K$,
we deduce that $K$ is Gorenstein.
\item Since any localization of the Gorenstein DG-ring $K$ is Gorenstein,
we may use Proposition \ref{prop:kosLocal} to reduce to the case where $(A,\bar{\m})$ is a noetherian local DG-ring,
and moreover $\bar{a}_1,\dots,\bar{a}_n \in \bar{\m}$.
Let us denote by $\bar{\n}$ the maximal ideal of $\mrm{H}^0(K)$,
and let $\k = \mrm{H}^0(A)/\bar{\m} = \mrm{H}^0(K)/\bar{\n}$ be the residue field.
Since $K$ is Gorenstein,
it is a dualizing DG-module over itself,
so by Proposition \ref{prop:dualKosz},
the DG-module $\mrm{R}\opn{Hom}_A(K,A)$ is also a dualizing DG-module over $K$.
Hence, according to \cite[Theorem II]{FIJ},
there is an integer $j \in \mathbb{Z}$,
such that
\[
\k[j] \cong \mrm{R}\opn{Hom}_K(\k,\mrm{R}\opn{Hom}_A(K,A)) \cong \mrm{R}\opn{Hom}_A(\k,A),
\]
which implies by \cite[Theorem II]{FIJ} that $A$ is a dualizing DG-module over itself,
so that $A$ is Gorenstein.
\item This follows from (1), (2), and the fact that the localization of a Gorenstein DG-ring is Gorenstein.
\end{enumerate}
\end{proof}

\begin{rem}
As mentioned in the introduction,
this result generalizes the main result of \cite{AG}, 
and \cite[Theorem 4.9]{FJ} in two different ways.
First, we do not assume that $A$ is local,
and second, we allow $A$ to be a commutative DG-ring instead of a commutative ring.
\end{rem}

\begin{rem}
In contrast with Theorem \ref{thm:Gorenstein},
the converse of Theorem \ref{thm:main} is false.
Indeed, if $(A,\m)$ is any noetherian local ring,
and if $(a_1,\dots,a_n)$ is a system of parameters of $A$,
then $K = K(A;a_1,\dots,a_n)$ satisfies
\[
\mrm{H}^0(K) = A/(a_1,\dots,a_n)
\]
and, by the definition of the notion of a system of parameters,
the latter has Krull dimension zero. 
By \cite[Proposition 4.8]{ShCM}, 
this implies that $K$ is Cohen-Macaulay,
but of course in general $A$ need not be Cohen-Macaulay.
\end{rem}

\section{Applications to fibers of local homomorphisms}\label{sec:APP}

Recall that if $\varphi:(A,\m) \to (B,\n)$ is a local homomorphism between noetherian local rings,
then its fiber ring, 
is the ring $\opn{Fib}(\varphi) = A/\m \otimes_A B$.
One may derive this,
obtaining the homotopy fiber of $\varphi$,
defined as $\mrm{H}\opn{Fib}(\varphi) := A/\m \otimes^{\mrm{L}}_A B$.
This is a noetherian local DG-ring.

This discussion generalizes to the case where $B$ is a DG-ring.
If $(A,\m)$ is a noetherian local ring,
and $(B,\bar{\n})$ is a noetherian local DG-ring,
then a map of DG-rings $\varphi:A \to B$ is called local if the induced map $\mrm{H}^0(\varphi):A \to \mrm{H}^0(B)$ is a local homomorphism.
In that case,
the homotopy fiber of $\varphi$ is defined to be
$\mrm{H}\opn{Fib}(\varphi) := A/\m \otimes^{\mrm{L}}_A B$.
Again, this is a noetherian local DG-ring.

In \cite{AF},
Avramov and Foxby introduced the notion of a Gorenstein local homomorphism,
and showed that a map $\varphi:(A,\m) \to (B,\n)$ of finite flat dimension between noetherian local rings is Gorenstein 
if and only if the DG-ring $\mrm{H}\opn{Fib}(\varphi)$ is Gorenstein.
In particular, this is the case if $A$ and $B$ are both Gorenstein.

We now obtain, as a corollary of Theorem \ref{thm:main},
some analogues of this result in the Cohen-Macaulay case.

\begin{cor}\label{cor:hfib}
Let $(A,\m)$ be a regular local ring,
and let $(B,\bar{\n})$ be a Cohen-Macaulay local DG-ring with constant amplitude.
Let $\varphi:A \to B$ be a local homomorphism of DG-rings.
Then the homotopy fiber
\[
\mrm{H}\opn{Fib}(\varphi) = A/\m \otimes^{\mrm{L}}_A B
\]
is a Cohen-Macaulay DG-ring.
\end{cor}
\begin{proof}
Since $A$ is a regular local ring, 
we may find an $A$-regular sequence $a_1,\dots,a_n \in \m$,
such that $(a_1,\dots,a_n) = \m$.
It follows that there is an isomorphism
\[
A/\m \cong K(A;a_1,\dots,a_n)
\]
in $\cat{D}(A)$.
Hence,
\[
\mrm{H}\opn{Fib}(\varphi) = A/\m \otimes^{\mrm{L}}_A B \cong K(A;a_1,\dots,a_n) \otimes^{\mrm{L}}_A B \cong K(B;\bar{b}_1,\dots,\bar{b}_n),
\]
where we have set $\bar{b}_i = \mrm{H}^0(\varphi)(a_i) \in \bar{\n}$.
Since $B$ is Cohen-Macaulay,
by Theorem \ref{thm:main}, 
we deduce that $\mrm{H}\opn{Fib}(\varphi)$ is a Cohen-Macaulay DG-ring.
\end{proof}

In the special case where $B$ is a ring,
we obtain:
\begin{cor}
Let $\varphi:(A,\m) \to (B,\n)$ be a local homomorphism between noetherian local rings,
such that $A$ is regular and $B$ is Cohen-Macaulay.
Then the homotopy fiber
\[
\mrm{H}\opn{Fib}(\varphi) = A/\m \otimes^{\mrm{L}}_A B
\]
is a Cohen-Macaulay DG-ring.
\end{cor}

In geometric language,
this application may be stated as:

\begin{cor}
Let $f:X \to Y$ be a morphism of schemes,
and assume that $X$ is Cohen-Macaulay and that $Y$ is nonsingular.
Then the homotopy fiber of $f$ at every point is a Cohen-Macaulay DG-ring.
\end{cor}

Our next application is a generalization of the miracle flatness theorem.
Below, for a complex $M$ over a ring $A$, we denote by $\flatdim_A(M)$ the flat dimension of $M$ over $A$.
Recall that if $\varphi:(A,\m) \to (B,\n)$ is a local map between noetherian local rings,
if $A$ is regular and $B$ is Cohen-Macaulay,
then the miracle flatness theorem (\cite[Theorem 23.1]{Mat}),
states that $\varphi$ is flat if and only if 
\[
\dim(B) = \dim(A) + \dim(\opn{Fib}(\varphi)) = \dim(A) + \dim(B/\m B).
\]

\begin{cor}\label{cor:miracle}
Let $(A,\m)$ be a regular local ring,
let $(B,\bar{\n})$ be a Cohen-Macaulay local DG-ring with constant amplitude,
and suppose that $B^0$ is a noetherian ring.
Let $\varphi:A \to B$ be a local homomorphism of DG-rings.
Then there is an equality
\begin{gather*}
\flatdim_A(B) = \dim(A) - \dim(\mrm{H}^0(B)) + \dim\left(\mrm{H}^0(B)/\m\mrm{H}^0(B)\right) + \amp(B).
\end{gather*}
\end{cor}
\begin{proof}
As in the proof of Corollary \ref{cor:hfib},
let $a_1,\dots,a_n$ be an $A$-regular sequence that generates $\m$,
so that
\[
\mrm{H}\opn{Fib}(\varphi) = A/\m\otimes^{\mrm{L}}_A B \cong K(B;\bar{b}_1,\dots,\bar{b}_n),
\]
where $\bar{b}_i = \mrm{H}^0(\varphi)(a_i) \in \bar{\n}$.
Because $B^0$ is noetherian,
we may deduce from \cite[Proposition 5.5(F)]{AFDim} that
\[
\flatdim_A(B) = \sup\{j\mid \opn{Tor}^A_j(A/\m,B) \ne 0\},
\]
and by the above, this is exactly $\amp(K(B;\bar{b}_1,\dots,\bar{b}_n))$.
According to (\ref{eqn:ampKoszul}),
we have that
\[
\amp(K(B;\bar{b}_1,\dots,\bar{b}_n)) = n - \dim(\mrm{H}^0(B)) + \dim(\mrm{H}^0(B)/\m\mrm{H}^0(B)) - \inf(B),
\]
so the result follows from noticing that $n = \dim(A)$ and $\amp(B) = -\inf(B)$.
\end{proof}

In the particular case where $B$ is a ring,
we obtain:
\begin{cor}
Let $\varphi:(A,\m) \to (B,\n)$ be a local homomorphism between noetherian local rings,
such that $A$ is regular and $B$ is Cohen-Macaulay.
Then there is an equality
\[
\flatdim_A(B) = \dim(A) - \dim(B) + \dim(B/\m B).
\]
\end{cor}

If $A$ is a ring, and $M$ is a complex of $A$-modules, normalized so that $\sup(M) = 0$,
then clearly $\flatdim_A(M) \ge \amp(M)$.
Thus, $M$ is as flat as possible over $A$ exactly when $\flatdim_A(M) = \amp(M)$.
Our next corollary states that for Cohen-Macaulay local DG-rings $B$,
they are as flat as possible over a regular base if $\mrm{H}^0(B)$ is.

\begin{cor}\label{cor:miracle2}
Let $(A,\m)$ be a regular local ring,
let $(B,\bar{\n})$ be a Cohen-Macaulay local DG-ring with constant amplitude,
and suppose that $B^0$ is a noetherian ring.
Let $\varphi:A \to B$ be a local homomorphism of DG-rings,
and assume that the induced map $\mrm{H}^0(\varphi):A \to \mrm{H}^0(B)$ is flat.
Then $\flatdim_A(B) = \amp(B)$.
\end{cor}
\begin{proof}
Since $A \to \mrm{H}^0(B)$ is flat, 
we know that
\[
\dim(A) - \dim(\mrm{H}^0(B)) + \dim\left(\mrm{H}^0(B)/\m\mrm{H}^0(B)\right) = 0,
\]
so by Corollary \ref{cor:miracle}, 
we see that $\flatdim_A(B) = \amp(B)$.
\end{proof}

Our final application is a DG generalization of the following classical characterization of Cohen-Macaulay local rings:
if a local ring $B$ contains a regular local ring $A$,
such that $B$ is finite over $A$,
then $B$ is Cohen-Macaulay if and only if $B$ is free as an $A$-module.
Equivalently, if and only if $B$ is flat as an $A$-module,
which means that $\flatdim_A(B) = \amp(B)$.
This generalizes to the DG setting as follows:

\begin{cor}\label{cor:DGreg}
Let $(A,\m)$ be a regular local ring,
and let $(B,\bar{\n})$ be a noetherian local DG-ring with bounded cohomology and constant amplitude.
Let $\varphi:A \to B$ be a local homomorphism of DG-rings,
and assume that the induced map $\mrm{H}^0(\varphi):A \to \mrm{H}^0(B)$ is a finite injective homomorphism.
Then $\flatdim_A(B) = \amp(B)$ if and only if $B$ is Cohen-Macaulay.
\end{cor}
\begin{proof}
Suppose first that $\flatdim_A(B) = \amp(B)$.
As $A$ is a regular local ring, 
in particular it is Gorenstein,
so it is a dualizing complex over itself.
Hence, since $A \to \mrm{H}^0(B)$ is finite,
it follows from \cite[Proposition 7.5]{Ye1} that
$R:=\mrm{R}\opn{Hom}_A(B,A)$ is a dualizing DG-module over $B$.
Since $A \to \mrm{H}^0(B)$ is finite, and $B$ is noetherian,
we deduce that $B \in \cat{D}^{\mrm{b}}_{\mrm{f}}(A)$.
By \cite[Corollary 2.10.F]{AFDim},
this implies that the projective dimension of $B$ over $A$ is equal to $\flatdim_A(B) = \amp(B)$.
This in turn implies that $\amp(R) \le \amp(B)$,
which shows that $B$ is local-Cohen-Macaulay.
Since $B$ has a dualizing DG-module and is local-Cohen-Macaulay,
by \cite[Corollary 8.11]{ShCM},
the assumption that $\opn{Supp}(\mrm{H}^{\inf(B)}(B)) = \opn{Spec}(\mrm{H}^0(B))$ implies that 
$B$ is Cohen-Macaulay.
Conversely, suppose that $B$ is Cohen-Macaulay.
Since $A \to \mrm{H}^0(B)$ is finite,
by the proof of \cite[Lemma 7.8]{Ye1},
we may replace $B$ by a quasi-isomorphic commutative DG-algebra over $A$,
with the extra property that $B^0$ is noetherian.
The assumption that $A \to \mrm{H}^0(B)$ is both finite and injective implies that
$\dim(A) = \dim(\mrm{H}^0(B))$, and that $\dim(\mrm{H}^0(B)/\m\mrm{H}^0(B)) = 0$.
Hence, by Corollary \ref{cor:miracle}, 
we deduce that $\flatdim_A(B) = \amp(B)$.
\end{proof}

\textbf{Acknowledgments.}

The author would like to thank Sean Sather-Wagstaff for asking me if Corollary \ref{cor:kosRing} is true,
and Amnon Yekutieli for helpful remarks on a previous version of this manuscript,
and for asking me if a version of Corollary \ref{cor:DGreg} holds.
The author is thankful to an anonymous referee for several suggestions that helped significantly improving this manuscript.
This work has been supported by Charles University Research Centre program No.UNCE/SCI/022,
and by the grant GA~\v{C}R 20-02760Y from the Czech Science Foundation.

\end{document}